\documentclass[oneside, a4paper,reqno]{amsart}
\usepackage{pdfsync}
\usepackage{stmaryrd}
\usepackage{mathrsfs}
\usepackage{amsmath, amsthm, amscd, amssymb, latexsym, eucal}
\usepackage[all]{xy}
 \calclayout \makeatletter
 \calclayout \makeatletter
\def\serieslogo@{} \def\@setcopyright{} \makeatother
\makeatletter
\renewcommand*\env@matrix[1][c]{\hskip -\arraycolsep
  \let\@ifnextchar\new@ifnextchar
  \array{*\c@MaxMatrixCols #1}}
\makeatother

\numberwithin{equation}{section}
\newtheorem{thm}{Theorem}[section]
\newtheorem{cor}[thm]{Corollary}
\newtheorem{lem}[thm]{Lemma}
\newtheorem{prop}[thm]{Proposition}

\theoremstyle{definition}
\newtheorem{defn}[thm]{Definition}
\newtheorem{rem}[thm]{Remark}
\newtheorem{exam}[thm]{Example}

\newtheorem*{ackn}{Acknowledgment}

\newcommand{\lxr}{\longrightarrow}

\newcommand{\A}{\mathscr A}
\newcommand{\B}{\mathscr B}
\newcommand{\C}{\mathscr C}

\newcommand{\F}{\mathcal F}
\newcommand{\G}{\mathcal G}
\newcommand{\I}{\mathcal I}

\newcommand{\M}{\mathcal M}

\newcommand{\U}{\mathcal U}
\newcommand{\V}{\mathcal V}
\newcommand{\W}{\mathcal W}
\newcommand{\X}{\mathcal X}
\newcommand{\Y}{\mathcal Y}
\newcommand{\Z}{\mathcal Z}

\newcommand{\mL}{\mathsf{L}}
\newcommand{\mf}{\mathsf{F}}
\newcommand{\mg}{\mathsf{G}}
\newcommand{\mt}{\mathsf{T}}
\newcommand{\mh}{\mathsf{H}}
\newcommand{\mU}{\mathsf{U}}
\newcommand{\mz}{\mathsf{Z}}
\newcommand{\mc}{\mathsf{C}}

 \DeclareMathOperator{\inc}{\mathsf{inc}}
\DeclareMathOperator*{\Ker}{\mathsf{Ker}}
 \DeclareMathOperator*{\Image}{\mathsf{Im}}
\DeclareMathOperator*{\Coker}{\mathsf{Coker}}

 \DeclareMathOperator{\pd}{\mathsf{pd}}
\DeclareMathOperator*{\id}{\mathsf{id}}

  \DeclareMathOperator*{\gld}{\mathsf{gl.dim}}
  \DeclareMathOperator*{\maxx}{\mathsf{max}}
  
\DeclareMathOperator*{\Mod}{\mathsf{Mod}-\!}

\DeclareMathOperator*{\End}{\mathsf{End}}
 \DeclareMathOperator*{\smod}{\mathsf{mod}-\!}

\DeclareMathOperator*{\inj}{\mathsf{inj}}
\DeclareMathOperator*{\proj}{\mathsf{proj}}
\DeclareMathOperator*{\Inj}{\mathsf{Inj}}
\DeclareMathOperator*{\Proj}{\mathsf{Proj}}

 \newcommand{\Gproj}{\operatorname{\mathsf{Gproj}}\nolimits}

\DeclareMathOperator{\Hom}{\mathsf{Hom}}

\DeclareMathOperator*{\Ext}{\mathsf{Ext}}
\DeclareMathOperator*{\Tor}{\mathsf{Tor}}

  \DeclareMathOperator*{\op}{\mathsf{op}}

     \DeclareMathOperator*{\rep}{\mathsf{rep.dim}}

   \DeclareMathOperator*{\du}{\mathsf{D}}

\DeclareMathOperator*{\llength}{\mathsf{LL}}

\newcommand{\iden}{\operatorname{Id}\nolimits}

\newsavebox{\proofbox}
\savebox{\proofbox}{\begin{picture}(7,7)%
  \put(0,0){\framebox(7,7){}}\end{picture}}

\begin{document}


\title[]{On Artin Algebras Arising from Morita Contexts}

\author[Ch. Psaroudakis]{Edward L. Green and Chrysostomos Psaroudakis}
\address{Department of Mathematics, Virginia Tech, Blacksburg, VA 24061, USA}
\email{green@math.vt.edu}
\address{Department of Mathematics, University of Ioannina, 45110
Ioannina, Greece} 
\email{hpsaroud@cc.uoi.gr}

\date{\today}

\thanks{{\bf -} The second named author has been co-financed by the European Union (European Social Fund - ESF) and Greek national funds through the Operational Program "Education and Lifelong Learning" of the National Strategic Reference Framework (NSRF) - Research Funding Program: Heracleitus II. Investing in knowledge society through the European Social Fund.}



\begin{abstract}
We study Morita rings $\Lambda_{(\phi,\psi)}=\bigl(\begin{smallmatrix}
A & _AN_B \\
_BM_A & B
\end{smallmatrix}\bigr)$ in the context of Artin algebras from various perspectives. First we study covariantly finite, contravariantly finite, and functorially finite subcategories of the module category of a Morita ring when the bimodule homomorphisms $\phi$ and $\psi$ are zero. Further we give bounds for the global dimension of a Morita ring $\Lambda_{(0,0)}$, as an Artin algebra, in terms of the global dimensions of $A$ and $B$ in the case when both $\phi$ and $\psi$ are zero. We illustrate our bounds with some examples. Finally we investigate when a Morita ring is a Gorenstein Artin algebra and then we determine all the Gorenstein-projective modules over the Morita ring $\Lambda_{\phi,\psi}$ in case $A=N=M=B$ and $A$ an Artin algebra.
\end{abstract}

\maketitle

\setcounter{tocdepth}{1} \tableofcontents

\section{Introduction}
Morita contexts, also known as pre-equivalence data,  have been introduced by Bass in \cite{Bass_Morita}, see also \cite{Cohn},  in his exposition of the Morita Theorems on equivalences of module categories. Let $A$ and $B$ be unital associative rings. Recall that a {\em Morita context} over $A$, $B$, is a $6$-uple $\M = (A,N,M,B, \phi, \psi)$, where  ${_{B}}M_{A}$ is a $B$-$A$-bimodule,  ${_{A}}N_{B}$ is an $A$-$B$-bimodule, and $\phi \colon   M\otimes_{A}N \lxr B$ is a $B$-$B$-bimodule homomorphism, and $\psi \colon N\otimes_{B}M \lxr A$ is an $A$-$A$-bimodule homomorphism, satisfying the following associativity conditions,  $\forall m,m'\in M$, $\forall n,n'\in N\colon$ 
\[\phi(m\otimes n)m'=m\psi(n\otimes m') \ \ \ \ \ \text{and} \ \ \ \ \ n\phi(m\otimes n')=\psi(n\otimes m)n'
\]
Associated to any Morita context $\M$ as above, there is an associative ring, the {\em Morita ring} of $\M$, which incorporates all the information involved in the $6$-uple $\M$,  defined to be the formal $2\times 2$ matrix ring 
\[
\Lambda_{(\phi,\psi)}(\M) =
         \begin{pmatrix}
           A & _AN_B \\
           _BM_A & B \\
         \end{pmatrix} 
\]
where $\Lambda_{(\phi,\psi)}(\M) = A \oplus N \oplus M \oplus B$ as an abelian group, and the  formal matrix multiplication is given by
\[\begin{pmatrix}
           a & n \\
           m & b \\
         \end{pmatrix} \cdot \begin{pmatrix}
           a' & n' \\
           m' & b' \\
         \end{pmatrix} = \begin{pmatrix}
           a a' + \psi(n\otimes m')  & a n'+ n b' \\
           m a' + b m' & b b' + \phi(m\otimes n')  \\
         \end{pmatrix}
\]
The Morita ring of a Morita context, not to be confused with the notion of a (right or left) Morita ring appearing in Morita duality,  has been studied explicitly by various authors in ring, module, or representation, theoretic framework; in this connection we refer to the papers by Amitsur \cite{Amitsur}, Muller \cite{Muller}, Green \cite{Green}, Cohen \cite{Cohen}, Loustaunau \cite{Loustaunau}, and Buchweitz \cite{Buchweitz}, among others. We refer also to the classical textbooks \cite{Lam}, \cite{McConnel}, \cite{Rowen} for the terminology of Morita rings.

It should be noted that the  Morita rings  form an omnipresent class of rings, providing sources of many important examples and the proper conceptual framework for the study of many problems in several different contexts in ring theory. We describe briefly some important examples and situations where Morita rings are involved. 

Let $A$ be a ring and  $M_{A}$ be a right $A$-module. If $B = \End_{A}(M)$ is the endomorphism ring of $A$, then  viewing  $M$  as a $B$-$A$-bimodule and setting $N = \Hom_{A}(M,A)$, it is easy to see that there exist naturally induced bimodule homomorphisms $\phi$ and $\psi$ and a Morita context $\M = (A,M,N,B, \phi, \psi)$. Hence any pair $(A,M_{A})$, where $A$ is a ring and $M_{A}$ is a right $A$-module induces a Morita context. 

An important special case is when $M = eA$, where $e^{2} = e$ is an idempotent element of $A$. Clearly then $N = \Hom_{A}(M,A) = Ae$ and $B = eAe$, and the Morita ring $\Lambda_{(\phi,\psi)}(\M)$  takes the form
\[
\Lambda_{(\phi,\psi)}(\M) =
         \begin{pmatrix}
           A & Ae \\
           eA & eAe \\
         \end{pmatrix} 
\]  
On the other hand if $e^{2} = e \in A$ is an idempotent element of $A$ and $f = 1_{A}-e$,  then the Pierce decomposition of $A$ with respect to the idempotents $e, f$ induces a Morita context $\M(e,f) = (eAe, eA , fA, fAf, \alpha, \beta)$,  and the ring $A$ is isomorphic to  
\[
A \, \ \cong \, \Lambda_{(\phi,\psi)}(\M(e,f)) =
         \begin{pmatrix}
           eAe & eAf
            \\
           fAe & fAf \\
         \end{pmatrix} 
\]  
Note that since any Morita ring $\Lambda_{\phi,\psi}(\M)$ contains the idempotents 
$e = \bigl( \begin{smallmatrix}
  1_{A} &0\\  0 & 0
\end{smallmatrix} \bigr)$ and $f = 1_{\Lambda}-e = \bigl( \begin{smallmatrix}
  0 &0\\  0 & 1_{B}
\end{smallmatrix} \bigr)$, it is not difficult to see that there is a ring isomorphism $\Lambda_{\phi,\psi}(\M) \cong \Lambda_{\phi,\psi}(\M(e,f))$. It follows that any Morita ring is isomorphic to the Morita ring arising from the Pierce decomposition of a ring $A$ with respect to two orthogonal idempotents whose sum is the identity of $A$.  We mention that, as a consequence, any upper or lower triangular matrix ring is a Morita ring.  

As another important example, in a more general context, let $\C$ be an additive category and $X, Y$ be arbitrary objects of $\C$. We view   $M := \Hom_{\C}(X,Y)$ as an $A$-$B$-bimodule and $N := \Hom_{\C}(Y,X)$ as a $B$-$A$-bimodule in a natural way, where $A = \End_{\C}(X)$ and $B = \End_{\C}(Y)$. It is easy to see that there is a Morita context $\M = (A, M, N, B, \phi, \psi)$ and an isomorphism of rings 
\[
{\End}_{\C}(X\oplus Y) \ \cong \ \Lambda_{\phi,\psi}(\M)
\]     
i.e. Morita rings appear as endomorphism rings of a direct sum of objects in any additive category. This is  the universal example of a Morita ring since it is not difficult to see that any Morita ring arises in this way. On the other hand the above construction gives the well-known bijective correspondence $\mu \colon \C \lxr \M$ between pre-additive categories $\C$ with two objects $X,Y$ and Morita Contexts $\M = (A,N,M,B, \phi, \psi)$. Under this correspondence $\mu(X) = \End_{\C}(X) = A$, $\mu(Y) = \End_{\C}(Y) = B$, $M = \Hom_{\C}(X,Y)$, $N = \Hom_{\C}(Y,X)$ and the maps $\phi$ and $\psi$ are given by the composition of maps in $\C$. As a consequence the study of Morita rings subsumes the study of pre-additive categories with two objects.      
   
\, 

The examples and situations  mentioned above and the important role they play in various different contexts provide a strong motivation for studying Morita rings in a general context using homological and representation-theoretic tools.   Our main aim in this paper is to study Morita rings, mainly  in the context of Artin algebras,  concentrating mainly at representation-theoretic and homological aspects. 

The organization and the main results of the paper are as follows. 

In section $2$ we collect preliminary notions and results on Morita rings that will be useful throughout the paper and we fix notation. In particular we describe the module category over a Morita ring and also we analyze the connections with recollement situations between the involved module categories. In section 3 we describe the projective, injective and simple modules in case the Morita ring is an  Artin algebra. Using this description we characterize when the Morita ring  is selfinjective and then, as an immediate consequence of this we give an upper bound for the representation dimension of the Morita ring $\Lambda_{0,0}(\M)$ arising from the Morita context $\M$ where $A = B = M = N = \Lambda\colon $selfinjective and $\phi = 0 = \psi$.

In sections 4 and 5  we study finiteness conditions on subcategories of the module category of a Morita ring  
$\Lambda_{(\phi,\psi)}(\M)$ arising from a Morita context $\M = (A, N, M, B, \phi, \psi)$, and also we investigate the global  dimension of $\Lambda_{(\phi,\psi)}(\M)$, in the special case when the bimodule homomorphisms $\phi$ and $\psi$ are zero. One advantage for working in this setting is that the module categories $\Mod{A}$, $\Mod{B}$, $\Mod{\bigl(\begin{smallmatrix}
A & 0 \\
_BM_A & B
\end{smallmatrix}\bigr)}$, $\Mod{\bigl(\begin{smallmatrix}
A & _AN_B \\
0 & B
\end{smallmatrix}\bigr)}$ are fully embedded into the module category $\Mod{\Lambda_{(0,0)}(\M)}$ as functorially finite subcategories. More generally we show in section 4 that if $\U$ is a functorially finite subcategory of $\Mod{A}$ and $\V$ is a functorially finite subcategory of $\Mod{B}$, then under some additional conditions the subcategories $\U$ and $\V$ induce a functorially finite subcategory of $\Mod{\Lambda_{(0,0)}(\M)}$, thus generalizing some well known results  of the literature, see \cite{triangular}, in the setting of triangular matrix rings. In section $5$, after introducing the notion of an {\em $A$-tight} (resp. {\em $B$-tight}) {\em projective $\Lambda_{(0,0)}$-resolution}, we prove our first main result of the section, namely if the bimodule $M$ has a $B$-tight projective $\Lambda_{(0,0)}$-resolution and the bimodule $N$ has an $A$-tight projective $\Lambda_{(0,0)}$-resolution, then the global dimension of $\Lambda_{(0,0)}$ is bounded by the sum of the global dimensions of $A$ and $B$ plus one. Further we deal with  the case where either $M$ or $N$ does not have a tight projective $\Lambda_{(0,0)}$-resolution and we present some formulas for the global dimension of $\Lambda_{(0,0)}$. We provide examples which shows that the inequalities of our bounds are sharp and that the inequalities can be proper. 

The final section 6 is devoted to investigate when a Morita ring $\Lambda_{(\phi,\psi)}(\M)$, regarded as an Artin algebra, is Gorenstein and discuss applications of this result. More precisely, under the assumption that the adjoint pair of functors $(M\otimes_A-,\Hom_B(M,-))$ induces a quasi-inverse equivalence between the category of $A$-modules of finite projective dimension and the category of $B$-modules of finite injective dimension, and the adjoint pair of functors $(N\otimes_B-,\Hom_A(N,-))$ induces a quasi-inverse equivalence between the category of $B$-modules of finite projective dimension and the category of $A$-modules of finite injective dimension, 
we show that the algebra $\Lambda_{(\phi,\psi)}(\M)$ is Gorenstein. For example if $\Lambda$ is a Gorenstein Artin algebra then the matrix algebra 
$\Delta_{(\phi,\phi)}(\M)=\bigl(\begin{smallmatrix}
\Lambda & \Lambda \\
\Lambda & \Lambda
\end{smallmatrix}\bigr)$, arising from the Morita context $\M = (\Lambda, \Lambda, \Lambda, \Lambda, \phi, \phi)$, is Gorenstein. As a consequence we determine the Gorenstein-projective modules over the Artin algebra $\Delta_{(\phi,\phi)}(\M)$, under the assumption that $\Lambda$ is a Gorenstein Artin algebra.  

\smallskip

{\bf Conventions and Notations.} We compose morphisms in a given category in a diagrammatic order. For a ring $R$ we work usually with left $R$-modules and the corresponding category is denoted by $\Mod R$. Our additive categories are assumed to have finite direct sums and our subcategories are assumed to be closed under isomorphisms and direct summands. For all unexplained notions and results concerning the representation theory of Artin algebras we refer to the book \cite{ARS}.

\section{Preliminaries on Morita Rings}
Let $A$ and $B$ be rings, $_AN_B$ an $A$-$B$-bimodule, $_BM_A$ a $B$-$A$-bimodule, and $\phi \colon   M\otimes_{A}N \lxr B$ a $B$-$B$-bimodule homomorphism, and $\psi \colon N\otimes_{B}M \lxr A$ an $A$-$A$-bimodule homomorphism. Then from the Morita context $\M = (A,N,M,B, \phi, \psi)$ we define the \textsf{Morita ring}$\colon$
\[
\Lambda_{(\phi,\psi)}(\M)=
         \begin{pmatrix}
           A & _AN_B \\
           _BM_A & B \\
         \end{pmatrix} 
\]
where the addition of elements of $\Lambda_{(\phi,\psi)}$ is componentwise and multiplication is given by
\[
         \begin{pmatrix}
           a & n \\
           m & b \\
         \end{pmatrix}
       \cdot 
         \begin{pmatrix}
           a' & n' \\
           m' & b' \\
         \end{pmatrix}=
         \begin{pmatrix}
           aa'+\psi(n\otimes m') & an'+nb' \\
           ma'+bm' & bb'+\phi(m\otimes n') \\
         \end{pmatrix}
\]
We assume that $\phi(m\otimes n)m'=m\psi(n\otimes m')$ and $n\phi(m\otimes n')=\psi(n\otimes m)n'$ for all $m,m'\in M$ and $n,n'\in N$. This condition ensures that $\Lambda_{(\phi,\psi)}(\M)$ is an associative ring. From now on we will write for simplicity $\Lambda_{(\phi,\psi)}$ instead of $\Lambda_{(\phi,\psi)}(\M)$.

\begin{rem}
Morita rings have appeared in the literature under various names, for instance$\colon$ the ring of the Morita context \cite{McConnel} and generalized matrix rings \cite{Green}, \cite{Palmer}.
\end{rem}

Since we are interested in Artin algebras, the following easy result characterizes when a Morita ring is an Artin algebra. 

\begin{prop}
Let $\Lambda_{(\phi,\psi)}=\bigl(\begin{smallmatrix}
A & _AN_B \\
_BM_A & B
\end{smallmatrix}\bigr)$ be a Morita ring. Then $\Lambda_{(\phi,\psi)}$ is an Artin algebra if and only if there is a commutative artin ring $R$ such that $A$ and $B$ are Artin $R$-algebras and $M$ and $N$ are finitely generated 
over $R$ which acts centrally both on $M$ and $N$.
\end{prop}

The description of the modules over a Morita ring $\Lambda_{(\phi,\psi)}$ is well known, see for instance \cite{Green}, but for completeness and due to our needs we include it also here. For this reason we introduce the following category. Let $\M(\Lambda)$ be the category whose objects are tuples $(X,Y,f,g)$
where $X\in \Mod{A}$, $Y\in \Mod{B}$, $f\in\Hom_B(M\otimes_AX,Y)$ and $g\in \Hom_A(N\otimes_BY,X)$ such that the following diagrams are commutative$\colon$
\[
\xymatrix{
  N\otimes_B M\otimes_A X \ar[d]_{\psi\otimes \iden_X} \ar[r]^{\ \ \ \ N\otimes f} &  N\otimes_BY \ar[d]^{g}     \\
  A\otimes_AX    \ar[r]^{\simeq} & X                  } \ \ \ \ \ \ \  \ \ \  \xymatrix{
  M\otimes_A N\otimes_B Y \ar[d]_{\phi\otimes \iden_Y} \ar[r]^{\ \ \ M\otimes g} &  M\otimes_AX \ar[d]^{f}     \\
  B\otimes_BY    \ar[r]^{\simeq} & Y                  }
\]
We denote by $\Psi_X$ and $\Phi_Y$ the following compositions$\colon$
\[
\xymatrix@C=0.5cm{
  N{\otimes_B}M{\otimes_A}X \ar[rr]^{ \ \ \ \ \psi\otimes \iden_X}  \ar @/^1.5pc/[rrrr]^{{\Psi_X}} && A\otimes_AX  \ar[rr]^{\simeq} && X  } \ \ \ \ \xymatrix@C=0.5cm{
  M\otimes_AN\otimes_BY \ar[rr]^{ \ \ \ \ \phi\otimes \iden_Y}  \ar @/^1.5pc/[rrrr]^{{\Phi_Y}} && B\otimes_BY  \ar[rr]^{\simeq} && Y  }
\]
Let $(X,Y,f,g)$ and $(X',Y',f',g')$ be objects of $\M(\Lambda)$. Then a morphism $(X,Y,f,g)\lxr (X',Y',f',g')$ in $\M(\Lambda)$ is a pair of homomorphisms $(a,b)$ where $a\colon X\lxr X'$ is an $A$-morphism and $b\colon Y\lxr Y'$ is a $B$-morphism such that the following diagrams are commutative$\colon$
\[
\xymatrix{
  M\otimes_A X \ar[d]_{M\otimes a} \ar[r]^{\ \ \ f} & Y \ar[d]^{b}     \\
  M\otimes_AX'    \ar[r]^{\ \ \ f'} & Y'                  } \ \ \ \ \ \ \  \ \ \  \xymatrix{
  N\otimes_B Y \ar[d]_{N\otimes b} \ar[r]^{\ \ \  g} &  X \ar[d]^{a}     \\
  N\otimes_BY'    \ar[r]^{\ \ \ g'} & X'                  }
\]
Dually since the functors $M{\otimes_A}-\colon \Mod{A}\lxr \Mod{B}$ and $N\otimes_B-\colon \Mod{B}\lxr \Mod{A}$ have right adjoints we can define the category $\widetilde{\M}(\Lambda)$. We denote by
\[
\xymatrix@C=0.5cm{
  \pi:\Hom_{B}(M\otimes_{A}X,Y) \ar[r]^{\simeq \ \   } & \Hom_{A}(X,\Hom_{B}(M,Y))
  } \ \text{and} \ \xymatrix@C=0.5cm{
  \rho:\Hom_{A}(N\otimes_{B}Y,X) \ar[r]^{\simeq \ \  } & \Hom_{B}(Y,\Hom_{A}(N,X))
  }
\]
adjoint isomorphims and let $\epsilon\colon M\otimes_{A}\Hom_B(M,-)\lxr \iden_{\Mod{B}}$, resp.
${\epsilon}'\colon N\otimes_B\Hom_A(N,-)\lxr \iden_{\Mod{A}}$, and
$\delta\colon \iden_{\Mod{A}}\lxr \Hom_B(M,M\otimes_A-)$, resp.
$\delta'\colon \iden_{\Mod{B}}\lxr \Hom_A(N,N\otimes_B-)$, be the counit and the unit of the adjoint pair $(M\otimes_A-,\Hom_A(M,-))$, resp. $(N\otimes_B-,\Hom_A(N,-))$. The objects of $\widetilde{\M}$ are tuples $(X,Y,\kappa,\lambda)$ where $X\in \Mod{A}$,
$Y\in \Mod{B}$, $\kappa\colon X\lxr \Hom_{B}(M,Y)$ and
$\lambda\colon Y\lxr \Hom_{A}(N,X)$ such that the following
diagrams are commutative$\colon$
\[
\xymatrix{
   X \ar[d]_{\kappa} \ar[r]^{\simeq \ \ \ \ \ \ } &  \Hom_{A}(A,X) \ar[d]^{}     \\
  \Hom_{B}(M,Y)    \ar[r]^{(M,\lambda) \ \ \ \ \ \ \ \ } & \Hom_{B}(M,\Hom_{A}(N,X))                  } \ \ \ \ \ \ \  \ \ \  \xymatrix{
   Y \ar[d]_{\lambda} \ar[r]^{\simeq \ \ \ \ \ \ } &  \Hom_{B}(B,Y) \ar[d]^{}     \\
  \Hom_{A}(N,X)    \ar[r]^{(N,\kappa) \ \ \ \ \ \ \ \ } & \Hom_{A}(N,\Hom_{B}(M,Y))
  }
\]
Let $(X,Y,\kappa,\lambda)$ and
$(X',Y',{\kappa}',{\lambda}')$ be objects in
$\widetilde{\M}(\Lambda)$. Then a morphism
$(X,Y,\kappa,\lambda)\lxr (X',Y',{\kappa}',{\lambda}')$ in
$\widetilde{\M}(\Lambda)$ is a pair of homomorphisms $(c,d)$ where
$c\colon X\lxr X'$ is an $A$-morphism and $d\colon Y\lxr Y'$ is a $B$-morphism such that the following diagrams are commutative$\colon$
\[
\xymatrix{
   X \ar[d]_{c} \ar[r]^{\kappa \ \ \ \ \ \ } & \Hom_{B}(M,Y) \ar[d]^{\Hom_{B}(M,d)}     \\
  X'    \ar[r]^{\kappa' \ \ \ \ \ \ } & \Hom_{B}(M,Y')                  } \ \ \ \ \ \ \  \ \ \ \xymatrix{
   Y \ar[d]_{d} \ar[r]^{\lambda \ \ \ \ \ \ } & \Hom_{A}(N,X) \ar[d]^{\Hom_{A}(N,c)}     \\
  Y'    \ar[r]^{\lambda' \ \ \ \ \ \ } & \Hom_{A}(N,X')                  } 
\]
We define the functor $\mf\colon \M(\Lambda)\lxr
\widetilde{\M}(\Lambda)$ by $\mf(X,Y,f,g)=(X,Y,\pi(f),\rho(g))$ on objects and
$\mf(a,b)=(a,b)$ on morphisms. Then it is straightforward that
$\mf\colon \M(\Lambda)\lxr \widetilde{\M}(\Lambda)$ is an isomorphism
of categories with inverse given by
$\mg(X,Y,\kappa,\lambda)=(X,Y,\pi^{-1}(\kappa),\rho^{-1}(\lambda))$ on
objects and $\mg(a,b)=(a,b)$ on morphisms. The relationship between $\Mod{\Lambda_{(\phi,\psi)}}$ and $\M(\Lambda)$ is given via the functor $\mf'\colon \M(\Lambda)\lxr \Mod{\Lambda_{(\phi,\psi)}}$ which is defined as follows. If $(X,Y,f,g)$ is an oblect of $\M(\Lambda)$, then we define $\mf'(X,Y,f,g)=X\oplus Y$ as abelian groups, and the $\Lambda_{(\phi,\psi)}$-module structure is given by
\[
         \begin{pmatrix}
           a & n \\
           m & b \\
         \end{pmatrix}(x,y)=(ax+g(n\otimes y), by+f(m\otimes x)) 
\]
for all $a\in A, b\in B, n\in N, m\in M, x\in X$ and $y\in Y$. One can verify easily that the object $\mf'(X,Y,f,g)$ is a $\Lambda_{(\phi,\psi)}$-module. If $(a,b)\colon (X,Y,f,g)\lxr (X',Y',f',g')$ is a morphism in
$\M$ then $\mf'(a,b)=a\oplus b\colon X\oplus Y\lxr X'\oplus Y'$.

\begin{prop}\label{tuples}
Let $\Lambda_{(\phi,\psi)}=\bigl(\begin{smallmatrix}
A & _AN_B \\
_BM_A & B
\end{smallmatrix}\bigr)$ be a Morita ring. Then the categories $\Mod{\Lambda_{(\phi,\psi)}}$, $\M(\Lambda)$ and $\widetilde{\M}(\Lambda)$ are equivalent.
\begin{proof}
See \cite[Theorem $1.5$]{Green}.
\end{proof}
\end{prop}

From now on we will identify the modules over $\Lambda_{(\phi,\psi)}$ with the objects of $\M(\Lambda)$. We define the following functors$\colon$
\begin{enumerate}
\item The functor $\mt_{A}\colon \Mod{A}\lxr \Mod{\Lambda_{(\phi,\psi)}}$ is defined by
$\mt_{A}(X)=(X,M\otimes_AX,\iden_{M\otimes X},\Psi_X)$ on the objects $X\in \Mod{A}$ and given an $A$-morphism $a\colon X\lxr X'$ then
$\mt_{A}(a)=(a,M\otimes a)$.

\item The functor $\mU_{A}\colon \Mod{\Lambda_{(\phi,\psi)}}\lxr \Mod{A}$ is defined by
$\mU_{A}(X,Y,f,g)=X$ on the objects $(X,Y,f,g)\in \Mod{\Lambda_{(\phi,\psi)}}$ and
given a morphism $(a,b)\colon (X,Y,f,g)\lxr (X',Y',f',g')$ in
$\Mod{\Lambda_{(\phi,\psi)}}$ then $\mU_{A}(a,b)=a$.

\item The functor $\mt_{B}\colon \Mod{B}\lxr \Mod{\Lambda_{(\phi,\psi)}}$ is defined by
$\mt_{B}(Y)=(N\otimes_BY,Y,\Phi_Y,\iden_{N\otimes Y})$ on the objects $Y\in \Mod{B}$ and given a $B$-morphism $b\colon Y\lxr Y'$ then $\mt_{B}(b)=(N\otimes b,b)$.

\item The functor $\mU_{B}\colon \Mod{\Lambda_{(\phi,\psi)}}\lxr \Mod{B}$ is defined by
$\mU_{B}(X,Y,f,g)=Y$ on the $\Lambda_{(\phi,\psi)}$-modules $(X,Y,f,g)$ and given a $\Lambda_{(\phi,\psi)}$-morphism $(a,b)\colon (X,Y,f,g)\lxr
(X',Y',f',g')$ then $\mU_{B}(a,b)=b$.

\item The functor $\mh_{A}\colon \Mod{A}\lxr \Mod{\Lambda_{(\phi,\psi)}}$ is defined by
$\mh_{A}(X)=(X,\Hom_A(N,X),\delta'_{M\otimes X}\circ \Hom_A(N,\Psi_X),\epsilon'_X)$ on the objects $X\in \Mod{A}$ and given an $A$-morphism $a\colon X\lxr X'$ then
$\mh_{A}(a)=(a,\Hom_A(N,a))$.

\item The functor $\mh_{B}\colon \Mod{B}\lxr \Mod{\Lambda_{(\phi,\psi)}}$ is defined by
$\mh_{B}(Y)=(\Hom_B(M,Y),Y,\epsilon_Y,\delta_{N\otimes Y}\circ \Hom_B(M,\Phi_Y))$ on the objects $Y\in \Mod{B}$ and given a $B$-morphism $b\colon Y\lxr Y'$ then
$\mh_{B}(b)=(\Hom_B(M,b),b)$.

\item Suppose that $\phi=0=\psi$. Then we define the functor $\mz_{A}\colon \Mod{A}\lxr \Mod{\Lambda_{(0,0)}}$ by $\mz_{A}(X)=(X,0,0,0)$ on the objects $X\in \Mod{A}$ and if
$a\colon X\lxr X'$ is an $A$-morphism then $\mz_{A}(a)=(a,0)$. Dually we define the functor $\mz_{B}\colon \Mod{B}\lxr \Mod{\Lambda_{(0,0)}}$.
\end{enumerate}

The following result gives more informations about these
functors and the module category over Morita rings. For the notion of recollements of abelian categories we refer to \cite{Pira}, \cite{recol}.

\begin{prop}\label{recollement}
\begin{enumerate}
\item The functors $\mt_{A}$, $\mt_{B}$ and $\mh_{A}, \mh_{B}$ are  fully faithful.

\item The pairs $(\mt_{A}, \mU_{A})$, $(\mt_{B}, \mU_{B})$ and $(\mU_{A}, \mh_{A})$, $(\mU_B, \mh_B)$ are adjoint pairs of functors.   

\item The functors $\mU_A$ and $\mU_{B}$ are exact.

\item We have $\Ker{\mU_A}=\Mod{\Lambda/\Lambda e_1 \Lambda}\simeq \Mod{{B/\Image{\phi}}}$, $\Ker{\mU_B}=\Mod{\Lambda/\Lambda e_2 \Lambda}\simeq \Mod{{A/\Image{\psi}}}$, $\Mod{e_1\Lambda e_1}=\Mod{A}$ and $\Mod{e_2\Lambda  e_2}=\Mod{B}$, where $e_1=\bigl(\begin{smallmatrix}
1 & 0 \\
0 & 0
\end{smallmatrix}\bigr)$ and $e_2=\bigl(\begin{smallmatrix}
0 & 0 \\
0 & 1
\end{smallmatrix}\bigr)$ are idempotent elements of $\Lambda_{(\phi,\psi)}$.

\item The following diagrams$\colon$ 
\[
 \ \ \ \ \ \ \ \xymatrix@C=0.2cm{
\Mod{{B/\Image{\phi}}} \ar[rrr]^{\inc} &&& \Mod{\Lambda_{(\phi,\psi)}} \ar[rrr]^{\mU_A } \ar @/_1.5pc/[lll]_{}  \ar
 @/^1.5pc/[lll]^{} &&& \Mod{A}
\ar @/_1.5pc/[lll]_{\mt_A} \ar
 @/^1.5pc/[lll]^{\mh_A}
 } \ \ \ \ \xymatrix@C=0.2cm{
\Mod{{A/\Image{\psi}}} \ar[rrr]^{\inc} &&& \Mod{\Lambda_{(\phi,\psi)}} \ar[rrr]^{\mU_B } \ar @/_1.5pc/[lll]_{}  \ar
 @/^1.5pc/[lll]^{} &&& \Mod{B}
\ar @/_1.5pc/[lll]_{\mt_B} \ar
 @/^1.5pc/[lll]^{\mh_B}
 }
\] 
are recollements of abelian categories. If $\phi=0=\psi$ then we have the following recollements$\colon$
\[
 \ \ \ \ \ \ \ \ \xymatrix@C=0.4cm{
\Mod{B} \ar[rrr]^{\mz_B \ \ } &&& \Mod{\Lambda_{(0,0)}} \ar[rrr]^{\mU_A} \ar @/_1.5pc/[lll]_{}  \ar
 @/^1.5pc/[lll]^{} &&& \Mod{A}
\ar @/_1.5pc/[lll]_{\mt_A} \ar
 @/^1.5pc/[lll]^{\mh_A}
 } \ \ \ \ \xymatrix@C=0.4cm{
\Mod{A} \ar[rrr]^{\mz_A \ \ } &&& \Mod{\Lambda_{(0,0)}} \ar[rrr]^{\mU_B} \ar @/_1.5pc/[lll]_{}  \ar
 @/^1.5pc/[lll]^{} &&& \Mod{B}
\ar @/_1.5pc/[lll]_{\mt_B} \ar
 @/^1.5pc/[lll]^{\mh_B}
 }
\]

\end{enumerate}
\begin{proof}
Parts (i)-(iv) are easy to verified, see also \cite{Psaroud}. For (v) recall that the recollement situation of $\Mod{\Lambda_{(\phi,\psi)}}$ means that $(\mt_A, \mU_A)$, $(\mU_A, \mh_A)$ are adjoint pairs, $\mt_A$, $\mh_A$ and $\inc$ are fully faithful and $\Mod{{B/\Image{\phi}}}$ is equivalent with $\Ker{\mU_A}=\{(0,Y,0,0) \ | \ Y\in \Mod{B}\}$. But this holds from (i)-(iv) and similarly we deduce that the rest diagrams are recollements of abelian categories as well. Note that the unlabeled functors in each diagram are induced from the counit of the adjoint pair $(\mt_A,\mU_A)$, resp. $(\mt_B,\mU_B)$, and the unit of the adjoint pair $(\mU_A,\mh_A)$, resp. $(\mU_B,\mh_B)$. This result is a special case of a more general statement in \cite{Psaroud}, see also \cite{recol}.
\end{proof}
\end{prop}

The next result describes the Morita rings $\Lambda_{(\phi,\psi)}$ with $\phi=0=\psi$. For the notion of trivial extension of rings we refer to \cite{FGR}.

\begin{prop}\label{trivial}\cite{FGR}
Let $\Lambda_{(0,0)}=\bigl(\begin{smallmatrix}
A & _AN_B \\
_BM_A & B
\end{smallmatrix}\bigr)$ be a Morita ring where the bimodule morphisms $\phi$ and $\psi$ are zero. Then we have an isomorphism of rings$\colon$
\[
\xymatrix@C=0.5cm{
  \Lambda_{(0,0)} \ar[r]^{\simeq \ \ \ \ \ \ \ \ \ } & (A\times B)\ltimes M\oplus N  }
\]
where $(A\times B)\ltimes M\oplus N$ is the trivial
extension ring of $A\times B$ by the $(A\times B)$-$(A\times
B)$-bimodule $M\oplus N$.
\begin{proof}
First we have to show that the abelian group $M\oplus N$ is an
$(A\times B)$-$(A\times B)$-bimodule. We define the morphisms$\colon$
\[
(A\times B)\times M\oplus N\lxr M\oplus N, \
[(a,b),(m,n)]\mapsto (bm,an)
\]
and
\[
M\oplus N\times (A\times B)\lxr M\oplus N, \
[(m,n),(a,b)]\mapsto (ma,nb)
\]
Then it is easy to establish that $M\oplus N$ is a left
$A\times B$-module and a right $A\times B$-module. Also since
${_AN_B}$ and ${_BM_A}$ are bimodules it follows that 
$M\oplus N$ is an $(A\times
B)$-$(A\times B)$-bimodule. Hence we can define the trivial
extension $(A\times B)\ltimes M\oplus N$ with elements
$[(a,b),(m,n)]$ where $(a,b)\in A\times B$, $(m,n)\in M\oplus N$,
the addition is componentwise and the multiplication is given by
\begin{eqnarray}
 [(a_1,b_1),(m_1,n_1)]\cdot [(a_2,b_2),(m_2,n_2)] &=& [(a_1,b_1)\cdot (a_2,b_2),(a_1,b_1)\cdot (m_2,n_2)+(m_1,n_1)\cdot (a_2,b_2)]  \nonumber \\
 &=& [(a_1a_2,b_1b_2),(b_1m_2,a_1n_2)+(m_1a_2,n_1b_2)] \nonumber \\
 &=& [(a_1a_2,b_1b_2),(b_1m_2+m_1a_2,a_1n_2+n_1b_2)] \nonumber
\end{eqnarray}
Then it is straightforward that the morphism $\Lambda_{(0,0)} \lxr (A\times B)\ltimes M\oplus N, \ \bigl(\begin{smallmatrix}
a & n \\
m & b
\end{smallmatrix}\bigr)\mapsto [(a,b),(m,n)]$ is an isomorphism of rings.
\end{proof}
\end{prop}

\subsection{Examples of Morita Rings} 
We continue by giving a variety of examples of Morita rings which will be used throughout the paper. 

The first example shows that any ring with an idempotent element is a Morita ring.

\begin{exam}
Let $R$ be a ring with an idempotent element $e$. Then from the Pierce decomposition of $R$ with respect to the idempotents $e, f=1_{R}-e$ it follows that $R$ is the Morita ring with $A=eRe$, $B=(1-e)R(1-e)$, $N=eR(1-e)$, $M=(1-e)Re$ and the bimodule homomorphisms $\phi$, $\psi$ are induced by the multiplication in $R$.
\end{exam}

The following example is well known from Morita equivalence.

\begin{exam}
Let $A$ be a ring and $P$ be an $A$-module. Then we have the following Morita ring$\colon$ 
\[
\Lambda_{(\phi,\psi)}=\begin{pmatrix}
A & {\Hom_{A}(P,A)} \\
P & \End_{A}(P)
\end{pmatrix}
\]
with bimodule homomorphisms $\phi\colon P\otimes_A\Hom_{A}(P,A)\lxr \End_{A}(P)$, $p\otimes f\mapsto \phi(p\otimes f)(p')=pf(p')$ and \\ $\psi\colon \Hom_{A}(P,A)\otimes_{\End_{A}(P)}P\lxr A$, $f\otimes p\mapsto \psi(f\otimes p)=f(p)$. Hence any pair $(A,P_{A})$, where $A$ is a ring and $P_{A}$ is a right $A$-module induces a Morita ring. Also it is well known that if the $A$-module $P$ is progenerator, then the rings $A$ and $\End_{A}(P)$ are Morita equivalent. 
\end{exam}

The next example shows that Morita rings are special cases of semi-trivial extensions \cite{Palmer}.

\begin{exam}
Let $R$ be a ring, $M$ a $R$-$R$-bimodule and $\theta\colon M\otimes_R M\lxr M$ a $R$-$R$-bimodule homomorphism. Then on the Cartesian product $R\times M$ we define multiplication by $(r, m)\cdot (r', m')=(rr'+\theta(m,m'), rm'+mr')$ such that $\theta(m\otimes m')m''=m\theta(m'\otimes m'')$, for every $r, r'\in R$ and $m, m', m''\in M$. Then this data defines a ring structure with unit element on the Cartesian product set $R\times M$. This ring is denoted by $R\ltimes_{\theta}M$ and is called the \textsf{semi-trivial extension} of $R$ by $M$ and $\theta$. We refer to Palmer \cite{Palmer} for more details.

Let $\Lambda_{(\phi,\psi)}=\bigl(\begin{smallmatrix}
A & _AN_B \\
_BM_A & B
\end{smallmatrix}\bigr)$ be a Morita ring. Set $R=A\times B$ and consider the  $R$-$R$-bimodule $\widetilde{M}=N\times M$. Then the map $\theta\colon \widetilde{M}\otimes_R\widetilde{M}\lxr R$ is a $R$-$R$-bimodule homomorphism, where $\widetilde{M}\otimes_R\widetilde{M}=N{\otimes_B}M\times N{\otimes_B}M$, and we have the following ring isomorphism$\colon$ $R\ltimes_{\theta}M\simeq \Lambda_{(\phi,\psi)}$.
\end{exam}

The following is the example which motivated this work. 

\begin{exam}
Let $\Lambda$ be an Artin algebra and $U, V$ two finitely generated $\Lambda$-modules. Consider the endomorphism Artin algebra $\End_{\Lambda}(U\oplus V)$. Then we have the Artin algebra$\colon$
\[
\Lambda_{(\phi,\psi)}={\End}_{\Lambda}(U\oplus V)\simeq
         \begin{pmatrix}
           \End_{\Lambda}(U) & \Hom_{\Lambda}(U,V) \\
           \Hom_{\Lambda}(V,U) & \End_{\Lambda}(V) \\
         \end{pmatrix} 
\]
where the bimodule homomorphisms $\phi\colon \Hom_{\Lambda}(V,U)\otimes \Hom_{\Lambda}(U,V)\lxr \End_{\Lambda}(V) $ and $\psi\colon \Hom_{\Lambda}(U,V)\otimes \Hom_{\Lambda}(V,U)\lxr \End_{\Lambda}(U)$ are given by composition.
\end{exam}

\begin{exam}
Let $\Lambda$ be an Artin algebra and let $\{\epsilon_1, \ldots, \epsilon_n\}$ be a full set of primitive orthogonal idempotents. Suppose $n\geq 2$ and $1\leq r< n$. Set $e_1=\sum_{i=1}^{r}\epsilon_i$ and $e_2=\sum_{i=r+1}^n\epsilon_i$. Then
\[
\Lambda_{(\phi,\psi)}=
         \begin{pmatrix}
           e_1\Lambda e_1 & e_1\Lambda e_2 \\
           e_2\Lambda e_1 & e_2\Lambda e_2 \\
         \end{pmatrix} 
\]
is a Morita ring having the structure of an Artin algebra and the bimodule homomorphisms $\phi\colon e_2\Lambda e_1\otimes_{e_1\Lambda e_1} e_1\Lambda e_2\lxr e_2\Lambda e_2$ and $\psi\colon e_1\Lambda e_2\otimes_{e_2\Lambda e_2}e_2\Lambda e_1\lxr e_1\Lambda e_1$ are given by multiplication. 
\end{exam}

\begin{exam}
Let $A$ be a ring and let $I$ and $J$ be ideals in $A$.
Then we have the Morita ring 
\[
\Lambda_{(\phi,\psi)} = \begin{pmatrix}A&J\\I&A\end{pmatrix}
\]
where $\phi\colon J\otimes_{A}I\lxr A$ and $\psi\colon I\otimes_AJ\lxr A$ are multiplication maps. Interesting special cases are when $I=J$, $A$ is an Artin algebra and $I$ and $J$ are contained in the Jacobson
radical of $A$, or $I=J=A$.
\end{exam}

Later in the paper we will study the special case of Morita rings with $\phi=\psi=0$.  More generally, the case where  $\phi=\psi$ is of interest.  Towards this end, we have the following result and its corollary.

\begin{lem}  Suppose that 
\[
\Lambda_{(\phi,\psi)} = \begin{pmatrix}A&N\\A&A\end{pmatrix}
\]  
is a Morita ring and that $\alpha\colon N\lxr A\otimes_AN$ and  $\beta\colon N\lxr N\otimes_AA$ are given by $\alpha(n)=1\otimes n$ and $\beta(n)=
n\otimes 1$. Then $\phi\circ\alpha=\psi\circ\beta$.
\end{lem} 
\begin{proof} We note that $\phi(a\otimes n)a'=a\psi(n\otimes a')$ for all
$a,a'\in A$ and $n\in N$.  Taking $a=a'=1_A$, we see that $\phi(\alpha(n))=\phi(1\otimes n)=\psi(n\otimes 1)=\psi(\beta(n))$ and the result follows.
\end{proof}

As a consequence, we have the following result.

\begin{cor}\label{equal}  Suppose that 
\[
\Lambda_{(\phi,\psi)} = \begin{pmatrix} A & A \\ A & A \end{pmatrix}
\] 
is a Morita ring. Then $\phi=\psi$.
\end{cor}
\begin{proof}
Noting that $A\otimes_AA$ is generated by $1\otimes 1$,  we need only to
show that $\phi(1\otimes 1)=\psi(1\otimes 1)$.
But, keeping the notation of the Lemma $2.12$, with $N=A$,
$\phi(1\otimes 1)=\phi(\alpha(1))=\psi(\beta(1))=\psi(1\otimes 1)$ and
we are done.
\end{proof}

\section{Projective, Injective and Simple Modules} 
In this section we describe the projective, injective
and simple modules over $\Lambda_{(\phi,\psi)}$ as an Artin algebra and we examine also when $\Lambda_{(\phi,\psi)}$ is selfinjective. We work in the setting of finitely generated modules over the Artin algebra $\Lambda_{(\phi,\psi)}$ although there is a description for arbitrary modules over Morita rings. We refer to \cite{Psaroud} for more details.

\subsection{Projective and Injective $\Lambda_{(\phi,\psi)}$-modules}
We start with the next result which gives a description of the indecomposable projective $\Lambda_{(\phi,\psi)}$-modules.

\begin{prop}\label{proj}
Let $\Lambda_{(\phi,\psi)}$ be a Morita ring regarded as an Artin algebra.
Then the indecomposable projective $\Lambda_{(\phi,\psi)}$-modules
are objects of the form$\colon$
\[
\left\{
  \begin{array}{lll}
   \mt_A(P)=(P, {M\otimes_{A}P}, \iden_{M\otimes_{A}P}, \Psi_P)  & \hbox{} \\
           & \hbox{} \\
   \mt_B(Q)=({N\otimes_BQ}, Q, \Phi_{Q}, \iden_{N\otimes_{B}Q})  & \hbox{}
  \end{array}
\right.
\]
where $P$ is an indecomposable projective $A$-module and
$Q$ is an indecomposable projective $B$-module.
\begin{proof}
If $P=\bigl(\begin{smallmatrix}
A & 0 \\
M & 0
\end{smallmatrix}\bigr)$ and $Q=\bigl(\begin{smallmatrix}
0 & N \\
0 & B
\end{smallmatrix}\bigr)$ then $\Lambda_{(\phi,\psi)}\simeq P\oplus Q$ as left
$\Lambda_{(\phi,\psi)}$-modules. From Proposition \ref{tuples} it follows
that the object of $\M(\Lambda)$ which corresponds to the
$\Lambda_{(\phi,\psi)}$-module $P$ is $(A,{M\otimes_{A}A},
\iden_{M\otimes A},\Psi_{A})$. Also the tuple $(N\otimes B,B,\Phi_{B},\iden_{N\otimes B})$ is the object of $\M(\Lambda)$ corresponding to $Q$. Let $A=P_1\oplus \cdots
\oplus P_n$ be the decomposition of $A$ into a direct sum of
indecomposable projective $A$-modules. Then we have the following decomposition of $(A,{M\otimes_{A}A},\iden_{M\otimes A},\Psi_{A})\colon$
\[
(A,{M\otimes_{A}A},
\iden_{M\otimes A},\Psi_{A})\simeq
(P_1,{M\otimes_{A}P_1}, \iden_{M\otimes_A P_1},\Psi_{P_1}) \oplus \cdots \oplus (P_n,{M\otimes_{A}P_n},
\iden_{M\otimes_A P_n},\Psi_{P_n})
\]
and for every $1\leq i\leq n$ we have the isomorphism ${\End}_{\Lambda_{(\phi,\psi)}}(P_i,{M\otimes_{A}P_i},
\iden_{M\otimes_{A}P_i},\Psi_{P_i})\simeq
{\End}_{A}(P_i)$. Since the algebra ${\End}_{A}(P_i)$ is local it follows that $(P_i,{M\otimes_{A}P_i}, \iden_{M\otimes_{A}P_i},\Psi_{P_i})$ is an
indecomposable projective $\Lambda_{(\phi,\psi)}$-module, for all $1\leq i\leq n$. Similarly if $B=Q_1\oplus \cdots \oplus Q_m$ is the decomposition
of $B$ into the direct sum of indecomposable projective $B$-modules,
then we infer that $({N\otimes_{B}Q_i},Q_i,\Phi_{Q_i},\iden_{N\otimes_{B}Q_i})$ is an indecomposable projective $\Lambda_{(\phi,\psi)}$-module for all $1\leq i\leq m$. In this way we get all
indecomposable projective $\Lambda_{(\phi,\psi)}$-modules up to isomorphism.
\end{proof}
\end{prop}

Our aim now is to describe the injective $\Lambda_{(\phi,\psi)}$-modules. In order to do this we describe how the duality acts on the objects of $\widetilde{\M}(\Lambda)$ using the equivalence of categories$\colon$ $\widetilde{\M}(\Lambda) \stackrel{\simeq}{\lxr} \M(\Lambda) \stackrel{\simeq}{\lxr}\smod{\Lambda_{(\phi,\psi)}} 
$. First we identify the opposite algebra
$\Lambda_{(\phi,\psi)}^{\op}$ with the Morita ring
\[
\Lambda_{(\phi,\psi)}^{\op}\simeq 
         \begin{pmatrix}
           B^{\op} & _{B^{op}}N_{A^{\op}} \\
           _{A^{\op}}M_{B^{op}} & A^{\op} \\
         \end{pmatrix}, \ 
         \begin{pmatrix}
           a & n \\
           m & b \\
         \end{pmatrix}^{\op}\mapsto 
         \begin{pmatrix}
           b^{\op} & n \\
           m & a^{\op} \\
         \end{pmatrix}
\]
Let $(X,Y,f,g)\in \M(\Lambda)$. Then the object
$(X,Y,\pi(f),\rho(g))\in \widetilde{\M}(\Lambda)$ and applying the
duality we obtain the morphisms $\du (\pi(f))\in \Hom_{A^{\op}}(\du\Hom_{B}(M,Y),\du X)$ and $\du (\rho(g))\in \Hom_{B^{\op}}(\du\Hom_{A}(N,X),\du Y)$. Let $0 \lxr Y \lxr I_0 \lxr I_1 $ be an injective
  coresolution of $Y$. Since the functors $\du \Hom_{B}(M,-), \, M\otimes_{B^{\op}}\du(-)\colon \smod{B}\lxr
  \smod{A^{\op}}$ are right exact we have the following exact commutative diagram$\colon$
\[
\xymatrix{
  \du \Hom_{B}(M,I_1) \ar[d]_{\simeq} \ar[r]^{} & \du \Hom_{B}(M,I_0) \ar[d]_{\simeq} \ar[r]^{} & \du \Hom_{B}(M,Y) \ar@{-->}[d]_{\sigma}
  \ar[r]^{} & 0  \\
  M\otimes_{B^{\op}}\du I_1 \ar[r]^{} &  M\otimes_{B^{\op}}\du I_0 \ar[r]^{} &  M\otimes_{B^{\op}}\du Y \ar[r]^{} & 0
  }
\]
Hence the morphism $\sigma\colon M\otimes_{B^{\op}}\du Y\stackrel{\simeq}{\lxr} \du\Hom_{B}(M,Y)$ is an $A^{\op}$-isomorphism which is functorial in
$Y$. Similarly we obtain a $B^{\op}$-isomorphism
$\tau\colon N\otimes_{A^{\op}}\du X\stackrel{\simeq}{\lxr} \du \Hom_{A}(N,X)$ which is
functorial in $X$. Then we have the object $(\du Y,\du X,\sigma\circ\du(
\pi(f)),\tau\circ\du (\rho(g)))\in \M(\Lambda^{\op}) \
$ where the morphisms are the following compositions$\colon$
\[
\xymatrix@C=0.5cm{
  M\otimes_{B^{\op}}\du Y \ar[rr]_{\simeq \ \ }^{\sigma \ \ } && \du \Hom_{B}(M,Y) \ar[rr]^{ \ \  \  \ \ \ \ \du(
\pi(f))} && \du X } \ \ \ \ \xymatrix@C=0.5cm{
  N\otimes_{A^{\op}}\du X \ar[rr]_{\simeq \ \ }^{\tau \ \ } && \du \Hom_{A}(N,X) \ar[rr]^{ \ \  \  \ \ \ \ (\du
\rho(g))} && \du Y } 
\]
If $(a,b)\colon (X,Y,f,g)\lxr (X',Y',f',g')$ is a
morphism in $\M(\Lambda)$ then it is easy to check that
\[
(\du(b),\du(a))\colon (\du Y',\du X',\sigma'\circ\du(
\pi(f')),\tau'\circ\du (\rho(g')))\lxr  (\du Y,\du
X,\sigma\circ\du (\pi(f)),\tau\circ\du (\rho(g)))
\]
is a morphism in $\M(\Lambda^{\op})$ and in this way we
obtain a contravariant functor $\du\colon \M(\Lambda)\lxr
\M(\Lambda^{\op})$. Then from the following commutative diagram$\colon$
\[
\xymatrix{
  \M(\Lambda) \ar[d]_{\du} \ar[r]^{\simeq \ \ \ } & \smod{\Lambda_{(\phi,\psi)}} \ar[d]^{\du} \\
 \M(\Lambda^{\op}) \ar[r]^{\simeq \ \ \ } &  \smod{\Lambda_{(\phi,\psi)}^{\op}}  }
\]
where $\du\colon \smod{\Lambda_{(\phi,\psi)}}\lxr
\smod{\Lambda_{(\phi,\psi)}^{op}}$ is the usual duality, we infer
that the functor $\du\colon \M(\Lambda)\lxr \M(\Lambda^{\op})$ is a
duality. We are now ready to describe the injective
$\Lambda_{(\phi,\psi)}$-modules.

\begin{prop}\label{inj}
Let $\Lambda_{(\phi,\psi)}$ be a Morita ring regarded as an Artin algebra.
Then the indecomposable injective $\Lambda_{(\phi,\psi)}$-modules
are of the form$\colon$
\[
\left\{
  \begin{array}{ll}
   \mh_A(I)=(I,\Hom_{A}(N,I),\delta'_{M\otimes I}\circ \Hom_A(N,\Psi_I),\epsilon'_I)  & \hbox{} \\
     & \hbox{} \\
   \mh_B(J)=(\Hom_{B}(M,J),J,\epsilon_{J},\delta_{N\otimes J}\circ \Hom_B(M,\Phi_J)) & \hbox{}
  \end{array}
\right.
\]
where $I$ is an indecomposable injective $A$-module and $J$ is an indecomposable injective $B$-module.
\begin{proof}
This follows from the description of the duality $\du\colon \M(\Lambda)\lxr
\M(\Lambda^{\op})$ and Proposition \ref{proj}.
\end{proof}
\end{prop}

\subsection{Simple $\Lambda_{(\phi,\psi)}$-modules}
In this subsection we determine the simple $\Lambda_{(\phi,\psi)}$-modules. Note that the description of the simple $\Lambda_{(\phi,\psi)}$-modules follows from the recollement diagram of Proposition \ref{recollement} and is due to Kuhn, see \cite{Kuhn} for a general result about simple objects in recollements of abelian categories. For completeness we sketch the proof in our case. 

Let $X$ be an $A$-module. Then we have the following commutative diagram$\colon$
\[
\xymatrix{
  \mt_{A}(X) \ar@{->>}[dr]^{} \ar[rr]^{(\iden_{X},\mu_{X})}
                &  &    \mh_{A}(X)     \\
                & \mc_{A}(X) \ \ \ \ar@{>->}[ur]                 }
\]
where $\mc_{A}(X)=\Image{(\iden_{X},\mu_{X})}=(X,\Image{\mu_X},\kappa,\lambda)$ and
$\mu_X=\delta'_{M\otimes X}\circ \Hom_A(N,\Psi_X)$. Then the assigment $X\mapsto \Image{(\iden_{X},\mu_{X})}$ defines a functor $\mc_{A}\colon \smod{A}\lxr \smod{\Lambda_{(\phi,\psi)}}$ on objects and given an $A$-morphism $a\colon X\lxr X'$ then $\mc_{A}(a)=(a,\theta)$, where $\theta\colon \Image{\mu_{X}}\lxr \Image{\mu_{X'}}$ is the unique
morphism which makes the following diagram commutative$\colon$
\[
{\xymatrixrowsep{0.6pc}\xymatrixcolsep{0.6pc}
\xymatrix{ M\otimes_AX \ar[rrr]^{\mu_X} \ar@{->>}[rrd]_{} \ar[dddd]_{M\otimes a} &&& \Hom_A(N,X) \ar[dddd]^{\Hom_A(N,a)}   \\
           && \Image{\mu_X} \ar[dddd]^{\theta} \ar@{>->}[ur]  \\
            &&&  \\
            &&  \\
           M\otimes_A X'  \ar'[rr]^{ \ \ \ \ \ \ \ \ \ \mu_{X'}}[rrr] \ar@{->>}[rrd]_{} &&& \Hom_A(N,X')  \\
           && \Image{\mu_{X'}} \ar@{>->}[ur]  }}
\]
From the above diagram it follows that the functor $\mc_{A}\colon \smod{A}\lxr \smod{\Lambda_{(\phi,\psi)}}$ preserves epimorphisms and monomorphisms.  

We have the following result which shows that the functor $\mc_{A}$ lifts simple
modules.

\begin{lem}\label{lem1simples}
Let $S$ be a simple $A$-module. Then $\mc_{A}(S)$ is a simple $\Lambda_{(\phi,\psi)}$-module.
\begin{proof}
Consider the following exact commutative diagram$\colon$
\[
\xymatrix@C=0.5cm{
     & & & & \mt_{A}(S) \ar[rr]^{} \ar@{->>}[d]_{}  && \mh_{A}(S)  & &
          \\
  0 \ar[rr] && (X,Y,f,g) \ar[rr]^{} && \mc_{A}(S) \ar[rr]^{} \ar@{>->}[urr]  && (X',Y',f',g') \ar[rr] && 0
  }
\]
We claim that either $(X,Y,f,g)=0$ or $(X',Y',f',g')=0$. If we apply the
exact functor $\mU_{A}$ we get the short exact sequence
$0 \lxr X \lxr S \lxr X' \lxr 0$ and so $X=0$ or $X'=0$ since $S$ is simple. Hence $(X,Y,f,g)\in \Ker{\mU_A}$ or $(X',Y',f',g')\in \Ker{\mU_A}$. Let $(0,Y,0,0)$ be the submodule of $\mc_{A}(S)$. Then the object $(0,Y,0,0)$ is also a submodule of
$\mh_{A}(S)$ and $\Hom_{\Lambda_{(\phi,\psi)}}((0,Y,0,0),\mh_{A}(S))\simeq
\Hom_{A}(\mU_{A}(0,Y,0,0),S)=0$. This implies that $\mc_{A}(S)$ has no nonzero submodules. Similarly if $(0,Y',0,0)$ is the quotient module of $\mc_{A}(S)$ then $(0,Y',0,0)$ is also a quotient module of $\mt_{A}(S)$. But $\Hom_{\Lambda_{(\phi,\psi)}}{(\mt_{A}(S),(0,Y',0,0))}\simeq
\Hom_A{(S,\mU_{A}(0,Y',0,0))}=0$ and then $\mc_{A}(S)$ has no nonzero
quotient modules. We infer that $(X,Y,f,g)=0$ or
$(X',Y',f',g')=0$ and therefore the $\Lambda_{(\phi,\psi)}$-module $\mc_{A}(S)$ is simple.
\end{proof}
\end{lem}

\begin{lem}\label{lem2simples}
Let $(X,Y,f,g)$ be a simple $\Lambda_{(\phi,\psi)}$-module such that 
$\mU_{A}(X,Y,f,g)=X\neq 0$. Then $X$ is a simple $A$-module and $\mc_{A}\mU_{A}(X,Y,f,g)\simeq (X,Y,f,g)$.
\begin{proof}
From the following commutative diagram$\colon$
\[
\xymatrix{
  \mt_{A}(X) \ar@{->>}[d]_{} \ar[r]^{(\iden_{X},f) \ } \ar[dr]^{} & (X,Y,f,g) \ar[d]^{(\iden_{X},\delta'_Y\circ \Hom_A(N,g))} \\
  \mc_{A}\mU_{A}(X,Y,f,g) \ \ar@{>->}[r]^{} & \mh_{A}(X)   }
\]
we deduce that the map
$\mc_{A}(X)\lxr \Image{(\iden_{X},\delta'_Y\circ
\Hom_A(N,g))}$ is a monomorphism. But then since $(X,Y,f,g)$ is simple and $\mc_{A}(X)\neq 0$ it follows that
$\mc_{A}\mU_{A}(X,Y,f,g)\simeq (X,Y,f,g)$. Let $X^{\prime}$ be a submodule of $X$ and let $i\colon X'\lxr X$ be the inclusion map. Then since the functor $\mc_A$ preserves monomorphisms we have the monomorphism $\mc_{A}(i)\colon \mc_{A}(X') \lxr \mc_{A}\mU_{A}(X,Y,f,g)$ and therefore $\mc_{A}(X')=0$ or $\mc_{A}(X')\simeq (X,Y,f,g)$. Thus if we apply the functor $\mU_{A}$ we get $X'=0$ or $X'\simeq X$
and therefore we conclude that the $A$-module $\mU_{A}(X,Y,f,g)=X$ is simple.
\end{proof}
\end{lem}

The following result describes the simple $\Lambda_{(\phi,\psi)}$-modules in terms of simple $A$-modules, simple $B$-modules, simple ${B/\Image{\phi}}$-modules and simple ${A/\Image{\psi}}$-modules.

\begin{prop}\label{simples}
There are the following bijections$\colon$
\[
\xymatrix@C=0.4cm{
  \{\text{simple} \ {B/\Image{\phi}} \text{-modules} \} \ar@<-.7ex>[rr]_-{\mz_B} && \{\text{simple} \ \Lambda_{(\phi,\psi)} \text{-modules} \ \text{such that} \ \mU_{A}(X,Y,f,g)= 0 \} \ar@<-.7ex>[ll]_-{\mU_B}
  }
\]
\[
 \ \ \ \ \ \ \ \  \xymatrix@C=0.4cm{
  \{\text{simple} \ A \text{-modules} \} \ar@<-.7ex>[rr]_-{\mc_{A}} && \{\text{simple} \ \Lambda_{(\phi,\psi)} \text{-modules} \ \text{such that} \ \mU_{A}(X,Y,f,g)\neq 0 \} \ar@<-.7ex>[ll]_-{\mU_{A}}
  }     
\]  
\[
\xymatrix@C=0.4cm{
  \{\text{simple} \ {A/\Image{\psi}} \text{-modules} \} \ar@<-.7ex>[rr]_-{\mz_A} && \{\text{simple} \ \Lambda_{(\phi,\psi)} \text{-modules} \ \text{such that} \ \mU_{B}(X,Y,f,g)= 0 \} \ar@<-.7ex>[ll]_-{\mU_A}
  }
\]  
\[  \ \ \ \ \ \ \ \   \xymatrix@C=0.4cm{
  \{\text{simple} \ B \text{-modules} \} \ar@<-.7ex>[rr]_-{\mc_{B}} && \{\text{simple} \ \Lambda_{(\phi,\psi)} \text{-modules} \ \text{such that} \ \mU_{B}(X,Y,f,g)\neq 0 \} \ar@<-.7ex>[ll]_-{\mU_{B}}    
  }
\] 
\begin{proof}
The second and fourth bijections follow from Lemma \ref{lem1simples}, Lemma \ref{lem2simples} and their duals concerning the functor $\mc_B$. Suppose now that $(0,Y,0,0)$ is a simple $\Lambda_{(\phi,\psi)}$-module. We claim that $\mU_{B}(0,Y,0,0)=Y$ is
a simple $B$-module. Since $(0,Y,0,0)$ is a $\Lambda_{(\phi,\psi)}$-module we have $\Phi_Y=0$. Let $Y'$ be a non-zero submodule of $Y$ and $i\colon Y'\hookrightarrow Y$ the inclusion map. Since $\Phi_{Y'}\circ i=(M\otimes_AN\otimes_B i)\circ \Phi_Y=0$ it follows that $\Phi_{Y'}=0$. This implies that we have the monomorphism $(0,i)\colon (0,Y',0,0)\lxr (0,Y,0,0)$ and since $(0,Y',0,0)$ is not zero and $(0,Y,0,0)$ is simple we infer that the $\Lambda_{(\phi,\psi)}$-modules $(0,Y',0,0)$ and $(0,Y,0,0)$ are isomorphic. Then $Y'\simeq Y$ and therefore the $B$-module $Y$ is simple. Conversely starting with a simple $B$-module $Y$ with $\Phi_Y=0$ then we deduce that the object
$\mz_{B}(Y)=(0,Y,0,0)$ is simple. Then the first bijection follows and similarly we derive the third one.
\end{proof} 
\end{prop}

\begin{rem}
Let $\{S_1,\ldots , S_n\}$ be a complete set of simple $A$-modules and $\{S_1',\ldots , S_m'\}$ a complete set of simple $B$-modules. The simple ${B/\Image{\phi}}$-modules are the simple $B$-modules with the additional property that $\Phi_{S_1'}=\cdots=\Phi_{S_m'}=0$. Then from the first two bijections of Proposition $3.5$ the simple $\Lambda_{(\phi,\psi)}$-modules are of the form$\colon$ $\mc_A(S_1), \ldots, \mc_A(S_n), \mz_B(S_1'),\ldots, \mz_B(S_m')$. Similarly the simple ${A/\Image{\psi}}$-modules are the simple $A$-modules such that $\Psi_{S_1}=\cdots=\Psi_{S_n}=0$, and using the last two bijections of Proposition $3.5$ we get that the simple $\Lambda_{(\phi,\psi)}$-modules are the following objects$\colon$ $\mc_B(S_1'), \ldots, \mc_B(S_m'), \mz_A(S_1),\ldots, \mz_A(S_n)$. It is easy to check that these two descriptions essentially coincide.  
\end{rem}

\subsection{Selfinjective Artin Algebras}
After the complete description of the indecomposable projective and
injective $\Lambda_{(\phi,\psi)}$-modules we are interested to find conditions for
the Artin algebra $\Lambda_{(\phi,\psi)}$ to be selfinjective. Recall
that an Artin algebra $A$ is \textsf{selfinjective} if $A$ is an injective and projective $A$-module.

The following result gives a sufficient condition for a Morita ring to be selfinjective. 

\begin{prop}
Let $\Lambda_{(\phi,\psi)}$ be a Morita ring which is an Artin algebra. Assume that the adjoint pair of functors $(M\otimes_A-,\Hom_B(M,-))$ induces an equivalence between $\proj{A}$ and $\inj{B}$, and the adjoint pair of functors $(N\otimes_B-,\Hom_A(N,-))$ induces an equivalence between $\proj{B}$ and $\inj{A}$. Then the algebra $\Lambda_{(\phi,\psi)}$ is selfinjective.
\begin{proof}
Let $\mt_A(P)=(P,{M\otimes_{A}P}, \iden_{M\otimes_{A}P},\Psi_P)$ be an indecomposable projective $\Lambda_{(\phi,\psi)}$-module, where $P$ is an indecomposable projective $A$-module. Since the categories $\proj{A}$ and $\inj{B}$ are equivalent we have isomorphisms $a\colon M\otimes_AP \stackrel{\simeq}{\lxr} J $ and $b^{-1}\colon \Hom_B(M,J)\stackrel{\simeq}{\lxr} P $ for some injective $B$-module $J$, where $b=\delta_P\circ \Hom_B(M,a)$. Then we have the injective $\Lambda_{(\phi,\psi)}$-module $\mh_B(J)=(\Hom_{B}(M,J),J,\epsilon_J,\delta_{N{\otimes_B} J}\circ \Hom_B(M,\Phi_J))$ and we claim that the map $(b,a)\colon \mt_A(P)\lxr \mh_B(J)$ is an isomorphism of $\Lambda_{(\phi,\psi)}$-modules. It is straightforward that the map $(b,a)$ is an isomorphism since both the maps $b$ and $a$ are isomorphims. Hence we have only to prove that $(b,a)$ is a morphism in $\smod{\Lambda_{(\phi,\psi)}}$, i.e. we have to show that the following diagrams are commutative$\colon$
\[
\xymatrix{
  M\otimes_A P \ar[d]_{M\otimes b}^{\simeq} \ar[r]^{\iden_{M\otimes P}} & M\otimes_AP \ar[d]^{a}_{\simeq}     \\
  M\otimes_A\Hom_B(M,J)    \ar[r]_{ \ \ \ \ \ \ \ \ \simeq}^{ \ \ \ \ \ \ \ \ \ \epsilon_{J}} & J } \ \ \ \ \ \ \  \ \ \  \xymatrix{
  N\otimes_B M\otimes_AP \ar[d]_{N\otimes a}^{\simeq} \ar[rrr]^{\Psi_P} &&&  P \ar[d]^{b}_{\simeq}     \\
  N\otimes_B J    \ar[rrr]^{\delta_{N\otimes J}\circ \Hom_B(M,\Phi_J) \ \ \ \ } &&& \Hom_{B}(M,J)    }
\]
Clearly the first diagram is commutative and for the second we have
\begin{eqnarray}
 (N\otimes_Ba)\circ \delta_{N\otimes J}\circ \Hom_B(M,\Phi_J) &=& \delta_{N\otimes M\otimes P}\circ \Hom_B(M,M\otimes_AN\otimes_Ba)\circ \Hom_B(M,\Phi_J)  \nonumber \\
 &=&  \delta_{N\otimes M\otimes P}\circ \Hom_B(M,\Phi_{M\otimes P})\circ \Hom_B(M,a)  \nonumber \\
 &=& \delta_{N\otimes M\otimes P}\circ \Hom_B(M,M\otimes \Psi_P)\circ \Hom_B(M,a)   \nonumber  \\
 &=& \Psi_P\circ \delta_P\circ \Hom_B(M,a) \nonumber \\
 &=& \Psi_P \circ b \nonumber  
\end{eqnarray}
Thus $\mt_A(P)\simeq \mh_B(J)$ and so the indecomposable projective $\Lambda_{(\phi,\psi)}$-module $\mt_A(P)$ is injective. Moreover from Proposition \ref{proj} we have also the indecomposable projective $\Lambda_{(\phi,\psi)}$-module $\mt_B(Q)=({N\otimes_{B}Q},Q,\phi \otimes \iden_{Q},\iden_{N\otimes_{B}Q})$ for some indecomposable projective $B$-module $Q$. Then using the equivalence between $\proj{B}$ and $\inj{A}$ it follows as above that $\mt_B(Q)\simeq \mh_A(I)$ for some injective $A$-module $I$ and so $\mt_B(Q)$ is an injective $\Lambda_{(\phi,\psi)}$-module. Since every indecomposable projective $\Lambda_{(\phi,\psi)}$-module is injective we infer that the Artin algebra $\Lambda_{(\phi,\psi)}$ is selfinjective. 
\end{proof}
\end{prop}

\begin{exam}
Let $\Lambda $ be a selfinjective Artin algebra. Then from Proposition $3.7$ we get that $\bigl(\begin{smallmatrix}
\Lambda & \Lambda \\
\Lambda & \Lambda
\end{smallmatrix}\bigr)_{(\phi,\phi)}$ is a selfinjective Artin algebra. 
\end{exam}

The following example shows that the converse of Proposition $3.7$ is not true in general.

\begin{exam}  Let $K$ be a field and $K\mathcal Q$ be the path algebra
where $\mathcal Q$ is the quiver
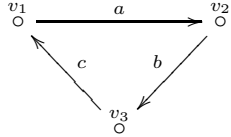
\[
\xymatrix{
\stackrel{v_1}{\circ}\ar^a[rr]&&\stackrel{v_2}{\circ}\ar_b[dl]\\
&\stackrel{v_3}{\circ}\ar_c[ul]
}
\]
Let $I$ be the ideal generated by $ba, cb$, and $ac$ and 
$\Lambda=K\mathcal Q/I$.   Then $\Lambda$ is a selfinjective 
finite dimensional $K$-algebra.  Setting $e=v_1$ and $e'=v_2+v_3$,
we view $\Lambda$ as a Morita ring via 
\[
\Lambda_{(\phi,\psi)} = \begin{pmatrix} e\Lambda e & e\Lambda e'\\
e'\Lambda e & e'\Lambda e'
\end{pmatrix}
\]
Note that, in this case, $\phi=\psi=0$. Since $e\Lambda e$ has one indecomposable projective-injective module up to isomorphism and $e'\Lambda e'$ has two nonisomorphic indecomposable projective-injective modules, the converse to Proposition $3.7$ fails. 
\end{exam}

For an Artin algebra $\Lambda$ we denote by ${\llength{(\Lambda)}}$ the Loewy length of $\Lambda$. We have the following consequence, where for the notion of Auslander's representation dimension we refer to \cite{Auslander:queen}.

\begin{cor}
Let $\Lambda$ be a selfinjective Artin algebra. Then$\colon$
\[
\rep{\begin{pmatrix}
           \Lambda & \Lambda \\
           \Lambda & \Lambda \\
         \end{pmatrix}_{(0,0)}
} \ \leq \ \ 2 \, {\llength{(\Lambda)}} 
\]
\begin{proof}
Since $\Lambda$ is selfinjective we have from Example $3.8$ that the matrix Artin algebra $\bigl(\begin{smallmatrix}
\Lambda & \Lambda \\
\Lambda & \Lambda
\end{smallmatrix}\bigr)_{(0,0)}$ is selfinjective. Then the result follows from 
\cite{Auslander:queen} since $\llength{\bigl(\begin{smallmatrix}
\Lambda & \Lambda \\
\Lambda & \Lambda
\end{smallmatrix}\bigr)_{(0,0)}}=2 \, \llength{(\Lambda)}$.
\end{proof}
\end{cor}

\section{Functorially Finite Subcategories}
In this section our purpose is to 
study finiteness conditions on subcategories of $\Mod{\Lambda_{(0,0)}}$. The reason for restricting to the case where $\phi=\psi=0$ is that we have full embeddings from the module categories $\Mod{A}$, $\Mod{B}$, $\Mod{\bigl(\begin{smallmatrix}
A & 0 \\
_BM_A & B
\end{smallmatrix}\bigr)}$ and $\Mod{\bigl(\begin{smallmatrix}
A & _AN_B \\
0 & B
\end{smallmatrix}\bigr)}$ to $\Mod{\Lambda_{(0,0)}}$. In particular we show that the above natural subcategories of $\Mod{\Lambda_{(0,0)}}$ are bireflective and therefore functorially finite. 

We start by defining the following full subcategories of $\Mod{\Lambda_{(0,0)}}$:
\[
\left\{
  \begin{array}{ll}
   \X = \big\{(X,Y,f,0) \ | \ f\colon M\otimes_AX\lxr Y \ \text{is an epimorphism} \big\}  & \hbox{} \\
    & \hbox{} \\
   \Y = \big\{(0,Y,0,0) \ | \ Y\in \Mod{B} \big\}=\Image{\mz_B}  & \hbox{} \\
    & \hbox{} \\
  \Z = \big\{(X,Y,0,g) \ | \ \rho(g)\colon Y\lxr \Hom_{A}(N,X) \ \text{is a monomorphism} \big\}  & \hbox{} \\
   & \hbox{} \\
  \X' = \big\{(X,Y,0,g) \ | \ g\colon N\otimes_BY\lxr X \ \text{is an epimorphism} \big\}  & \hbox{} \\
   & \hbox{} \\
   \Y' = \big\{(X,0,0,0) \ | \ X\in \Mod{A} \big\}=\Image{\mz_A}  & \hbox{} \\
   & \hbox{} \\
  \Z' = \big\{(X,Y,f,0) \ | \ \pi(f)\colon X\lxr \Hom_{B}(M,Y) \ \text{is a monomorphism} \big\}  & \hbox{}
  \end{array}
\right.
\]
 
We will show that the above subcategories have a special structure in $\Mod{\Lambda_{(0,0)}}$. First let us recall the notion of torsion pairs for abelian categories. Let $\B$ be an abelian category. Then a \textsf{torsion pair} in $\B$ is a pair $(\U,\V)$
of strict full subcategories of $\B$ satisfying the following
conditions$\colon$
\begin{enumerate}
\item $\Hom_{\B}(\U,\V)=0$, i.e. $\Hom_{\B}(U,V)=0$ for all $U\in
\U$ and $V\in \V$.

\item For every object $B\in \B$ there exists a short exact sequence $0 \lxr U_B \stackrel{f_B}{\lxr} B \stackrel{g^B}{\lxr} V^B \lxr 0$ in $\B$ such that $U_B\in \U$ and $V^B\in \V$.
\end{enumerate}
In that case, $\U$ is called a \textsf{torsion class} and
$\V$ is called a \textsf{torsion-free class}. It is very easy to observe that for a torsion pair $(\U,\V)$ in $\B$ we have that $\U$ is closed under factors, extensions and coproducts and $\V$ is closed under subobjects, extensions and products. Moreover in this case we have $\U={^\bot{\V}}=\{B\in \B \ | \ \Hom_{\B}(B,\V)=0 \}$ and $\V={\U}^\bot=\{B\in \B \ | \ \Hom_{\B}(\U,B)=0 \}$. A triple $(\U,\V,\W)$ of strict full subcategories of $\B$ is called \textsf{torsion, torsion-free triple} or \textsf{TTF-triple} for short, if $(\U,\V)$ and $(\V,\W)$ are torsion pairs. Then $\V$ is called a \textsf{TTF-class}. We refer to \cite{BR} for more informations about torsion theories.

Recall also that a full subcategory $\C$ is called \textsf{contravariantly finite} in $\B$ if for any object $B$ in $\B$ there exists a map $f_B\colon C_B\lxr B$, where $C_B$ lies in $\C$, such that the induced map $\Hom_{\B}(\C,f_B)\colon \Hom_{\B}(\C,C_B)\lxr \Hom_{\B}(\C,B)$ is surjective. In this case the map $f_B$ is called a right $\C$-approximation of $B$. Dually we have the notions of covariant finiteness and left approximations. Then $\C$ is called \textsf{functorially finite} if it is both contravariantly and covariantly finite.

The following result gives the  precise structure of the subcategories $\X, \Y, \Z, \X^{\prime}, \Y^{\prime}$, $\Z^{\prime}$ in $\Mod{\Lambda_{(0,0)}}$.

\begin{prop}\label{ttf-triple}
Let $\Lambda_{(0,0)}$ be a Morita ring. Then the triples $(\X,\Y,\Z)$ and $(\X',\Y',\Z')$ are \textsf{TTF}-triples in
$\Mod{\Lambda_{(0,0)}}$.
\begin{proof}
Let $(X,Y,f,g)$ be a $\Lambda_{(0,0)}$-module. We write $f=\kappa\circ \lambda$ for the factorization of $f$ through the $\Image{f}$. Then from the counit of the adjoint pair $(\mt_A,\mU_A)$ we have the following exact sequence$\colon$
 \[
\xymatrix{
 0 \ar[r] & (0,\Ker{f},0,0) \ar[r] & {\mt}_{A}(X) \ar@{->>}[d]^{} \ar[r]^{(\iden_{X},f) \ \ \ \ }
                  &   (X,Y,f,g)  \ar[r] & (0,\Coker{f},0,0) \ar[r] & 0   \\
          & &      \Image{(\iden_{X},f)} \ar@{>->}[ur]  &     &    &      }
\]
where $\Image{(\iden_{X},f)}=(X,\Image{f},\kappa,(N\otimes_B \lambda)\circ g)$. But from the following commutative diagram$\colon$
\[
\xymatrix{
  N\otimes_BM\otimes_AX \ar[rr]^{N\otimes f} \ar@{->>}[d]_{N\otimes \kappa} \ar @/^1.5pc/[rrrr]^{{\Psi_X}=0} && N\otimes_BY  \ar[rr]^{g} && X  \\
               N\otimes_B\Image{f} \ar[urr]_{N\otimes \lambda}   &     &         }
\]
it follows that the map $({N\otimes_B \lambda})\circ g$ is zero since the map ${N\otimes_B\kappa}$ is an epimorphism. Therefore we have the exact sequence $0 \lxr (X,\Image{f},\kappa,0) \lxr (X,Y,f,g) \lxr (0,\Coker{f},0,0) \lxr 0$ with $(X,\Image{f},\kappa,0)\in \X$ and $(0,\Coker{f},0,0)\in \Y$. Let $(X,Y,f,0)\in \X$ and $(0,Y',0,0)\in \Y$. Then since the map $f\colon M\otimes_AX\lxr Y$ is an epimorphism we get that $\Hom_{\Lambda_{(0,0)}}((X,Y,f,0),(0,Y',0,0))=0$. Hence $(\X,\Y)$ is a torsion pair in $\Mod{\Lambda_{(0,0)}}$. We continue now in order to show that $(\Y,\Z)$ is also a torsion pair. Let $(X,Y,f,g)$ be a $\Lambda_{(0,0)}$-module. From the unit of the adjoint pair $(\mU_A,\mh_A)$ we have the following exact sequence$\colon$
\[
\xymatrix{
 0 \ar[r] & (0,\Ker{\rho{(g)}},0,0) \ar[r] & (X,Y,f,g) \ar@{->>}[d]^{} \ar[r]^{ \ \ (\iden_{X},\rho{(g)})}
                  &   \mh_A(X)  \ar[r] & (0,\Coker{\rho{(g)}},0,0) \ar[r] & 0   \\
          & &      \Image{(\iden_{X},\rho{(g)})}  \ar@{>->}[ur]  &     &    &      }
\]
where $\rho{(g)}=\delta'_Y\circ \Hom_A(N,g)$. We write $\rho(g)=\gamma \circ n$ for the factorization through the $\Image{\rho{(g)}}$. The image of the morphism $(\iden_{X},\rho{(g)})$ is the $\Lambda_{(0,0)}$-module $\Image{(\iden_{X},\rho{(g)})}=(X,\Image{\rho{(g)}},m,\rho^{-1}(n))$, where $m=f\circ \gamma$ and the map $\rho(\rho^{-1}(n))=n\colon \Image{\rho{(g)}}\lxr \Hom_A(N,X)$ is a monomorphism. But then the map $m=f\circ \gamma=\pi^{-1}(\pi(f)\circ \Hom_B(M,\gamma))=0$ since the following diagram is commutative$\colon$
\[
\xymatrix{
  X \ar[rr]^{\pi(f)  }  \ar @/^1.5pc/[rrrr]^{\Psi'_X=0} && \Hom_B(M,Y) \ar[d]_{\Hom_B(M,\gamma)}  \ar[rr]^{\Hom_B(M,\rho{(g)}) \ \ \ \ \ \ \ \ } && \Hom_B(M,\Hom_A(N,X))  \\
         & &      \Hom_B(M,\Image{\rho{(g)}}) \ \ \ \  \ \ \ \ar@{>->}[urr]_{\Hom_B(M,n)}    }
\]
Thus we obtain the exact sequence
$ 0 \lxr (0,\Ker{\rho{(g)}},0,0) \lxr (X,Y,f,g) \lxr (X,\Image{\rho{(g)}},0,\rho^{-1}(n)) \lxr 0 $ in $\Mod{\Lambda_{(0,0)}}$ with $(0,\Ker{\rho{(g)}},0,0)\in \Y$ and $(X,\Image{\rho{(g)}},0,\rho^{-1}(n))\in \Z$. Let $(0,b)\colon (0,Y',0,0)\lxr (X,Y,0,g)$ be a morphism in $\Mod{\Lambda_{(0,0)}}$ with $(0,Y',0,0)\in \Y$ and $(X,Y,0,g)\in \Z$. Then from the following commutative diagrams$\colon$ 
\[
\xymatrix{
  N\otimes_BY' \ar[d]_{N\otimes b} \ar[r]^{} & 0 \ar[d]^{}     \\
  N\otimes_BY    \ar[r]^{\ \ \ g} & X                  } \ \ \ \ \ \ \  \ \ \  \xymatrix{
   Y' \ar[d]_{ b} \ar[r]^{} &  0 \ar[d]^{}     \\
  Y    \ar[r]^{\rho{(g)} \ \ \ \ \ \ \ } & \Hom_A(N,X)                  }
\]  
it follows that $b=0$ since the map $\rho{(g)}$ is a monomorphism. Thus $\Hom_{\Lambda_{(0,0)}}(\Y,\Z)=0$ and therefore we conclude that  
$(\Y,\Z)$ is a torsion pair in $\Mod{\Lambda_{(0,0)}}$. Similarly 
we prove that $(\X',\Y',\Z')$ is a \textsf{TTF}-triple in $\Mod{\Lambda_{(0,0)}}$.
\end{proof}
\end{prop}

\begin{rem}\label{mapzero}
Let $(X,Y,f,g)$ be a $\Lambda_{(0,0)}$-module with $f\colon M\otimes_AX \lxr Y$ an epimorphism. Then the morphism $g\colon N\otimes_BY\lxr X$ is zero since the composition $N\otimes_BM\otimes_AX \lxr N\otimes_BY \lxr X $ is zero and the map ${N\otimes_Bf}$ is an epimorphism. We used this argument in the above proof and note that the same holds for the objects of $\Z$, $\X'$ and $\Z'$ as well.
\end{rem}

As a consequence of Proposition \ref{ttf-triple} and \cite{AR:homological} we have the following.

\begin{cor}
Let $\Lambda_{(0,0)}$ be a Morita ring.
\begin{enumerate}
\item The full subcategory $\X=\{(X,Y,f,0) \ | \ f\colon M\otimes_AX\lxr Y \ \text{is an epimorphism} \}$ is contravariantly finite in 
$\Mod{\Lambda_{(0,0)}}$, closed under extensions, quotients and coproducts, and $\mt_A(\Mod{A})\subseteq \X$.

\item The full subcategory $\X'=\{(X,Y,0,g) \ | \ g\colon N\otimes_BY\lxr X \ \text{is an epimorphism} \}$ is contravariantly finite
 in $\Mod{\Lambda_{(0,0)}}$, closed under extensions, quotients and coproducts, and $\mt_B(\Mod{B})\subseteq \X'$.

\item The full subcategory $\Z=\{(X,Y,0,g) \ | \ \rho(g)\colon Y\lxr \Hom_{A}(N,X) \ \text{is a monomorphism} \}$ is covariantly finite in 
$\Mod{\Lambda_{(0,0)}}$, closed under extensions, subobjects and products, and $\mh_A(\Mod{A})\subseteq \Z$. 

\item The full subcategory $\Z'=\{(X,Y,f,0) \ | \ \pi(f)\colon X\lxr \Hom_{B}(M,Y) \ \text{is a monomorphism} \}$ is covariantly finite
in $\Mod{\Lambda_{(0,0)}}$, closed under extensions, subobjects and products, and $\mh_B(\Mod{B})\subseteq \Z'$.
\end{enumerate}
\end{cor}

The next result describes the categories of modules over $A$ and $B$ via the subcategories $\X$, $\Z$, $\X'$, $\Z'$.

\begin{cor}
Let $\Lambda_{(0,0)}$ be a Morita ring. Then there is an equivalence
\[
\xymatrix@C=0.5cm{
  \Mod{A} \ar[r]^{\simeq} & \X\cap \Z } \ \ \ \text{and} \ \ \ \xymatrix@C=0.5cm{
  \Mod{B} \ar[r]^{\simeq \ } & \X'\cap \Z'}
\]  
\begin{proof}
From Proposition \ref{recollement} we have the recollement $(\Mod{B},\Mod{\Lambda_{(0,0)}},\Mod{A})$ and so the quotient category $\Mod{\Lambda_{(0,0)}}/\Mod{B}$ is equivalent with $\Mod{A}$. From Proposition \ref{ttf-triple} we have the \textsf{TTF}-triple $(\X,\Y,\Z)$ in $\Mod{\Lambda_{(0,0)}}$. Then from \cite[Proposition $1.3$]{BR} it follows that $\Mod{\Lambda_{(0,0)}}/\Y\simeq \X\cap \Z$. But since we can identify $\Y$ with $\Mod{B}$ it follows that $\Mod{A}$ is equivalent with $\X\cap \Z$. Similarly using the \textsf{TTF}-triple $(\X',\Y',\Z')$ and the recollement $(\Mod{A},\Mod{\Lambda_{(0,0)}},\Mod{B})$ we infer that $\Mod{B}$ is equivalent with $\X'\cap \Z'$.  
\end{proof}
\end{cor}

We denote by ${\X}_0$  the full subcategory of $\Mod{\Lambda_{(0,0)}}$ whose objects are the tuples $(X,Y,f,g)$ such that there is an exact sequence $0\lxr K_0 \lxr
\mt_A(P_0)\lxr (X,Y,f,g)\lxr 0$ with $P_0\in \Proj{A}$. Similarly we define the subcategories ${\Y}_0=\{(X,Y,f,g)\in \Mod{\Lambda_{(0,0)}} \ | \  \exists
\ 0 \lxr (X,Y,f,g) \lxr
\mh_A(I_0)\lxr L_0 \lxr 0 \ \text{exact with} \ I_0\in
\Inj{A} \}$, ${\X}'_0=\{(X,Y,f,g)\in \Mod{\Lambda_{(0,0)}} \ | \ \exists \ 0\lxr K_0 \lxr \mt_B(Q_0)\lxr (X,Y,f,g)\lxr 0 \ \text{exact with} \ Q_0\in \Proj{B} \}$ and ${\Y}'_0=\{(X,Y,f,g)\in \Mod{\Lambda_{(0,0)}} \ | \  \exists
\ 0 \lxr (X,Y,f,g) \lxr \mh_B(J_0)\lxr L_0 \lxr 0 \ \text{exact with} \ J_0\in
\Inj{B} \}$. 

The reason for defining the above subcategories is the following result which provides another description of the subcategories $\X,\Z,\X'$ and $\Z'$.

\begin{prop}
Let $\Lambda_{(0,0)}$ be a Morita ring. Then$\colon$
$\X=\X_0$, \ $\Z=\Y_0$, \ $\X'=\X'_0$ and $\Z'=\Y'_0$.
\begin{proof}
Let $(X,Y,f,0)\in \X$ and let $a\colon P\lxr X$ be an epimorphism with $P\in \Proj{A}$. Then we have the morphism $(a,Fa\circ f)\colon \mt_A(P)\lxr (X,Y,f,0)$ in $\Mod{\Lambda_{(0,0)}}$ which is an epimorphism since $f\colon M\otimes_AX\lxr Y$ is an epimorphism. Hence $(X,Y,f,0)\in \X_0$. Conversely, let $(X,Y,f,g)\in \X_0$. Since there exists an epimorphism $(a,b)\colon \mt_A(P)\lxr (X,Y,f,g)$ it follows that $b=Fa\circ f$ and $b$ is an epimorphism. Hence the map $f$ is an epimorphism and from Remark \ref{mapzero} it follows that the map $g=0$. Thus the object $(X,Y,f,0)\in \X$. Therefore we have $\X=\X_0$ and similarly we prove the other statements.   
\end{proof}
\end{prop}

For the next result we need to recall the notion of bireflective subcategories. A full subcategory $\C$ of an abelian category $\B$ is said to be \textsf{reflective} if the inclusion functor $\mathsf{i}\colon \C\lxr \B$ has a left adjoint. Dually the subcategory $\C$ is called \textsf{coreflective} if $\mathsf{i}\colon \C\lxr \B$ has a right adjoint.  
Then $\C$ is \textsf{bireflective} if it is both reflective and coreflective.  
We refer to \cite{Gabriel} and \cite{Geigle-Lenzing} for more information about bireflective subcategories.

The following gives the exact properties of the natural subcategories of $\Mod{\Lambda_{(0,0)}}$.

\begin{thm}\label{bireflective}
Let $\Lambda_{(0,0)}$ be a Morita ring. Then the full subcategories
\[
\Mod{A}, \ \Mod{B}, \ \Mod{\bigl(\begin{smallmatrix}
A & 0 \\
_BM_A & B
\end{smallmatrix}\bigr)}, \ \Mod{\bigl(\begin{smallmatrix}
A & _AN_B \\
0 & B
\end{smallmatrix}\bigr)}
\]
are bireflective in $\Mod{\Lambda_{(0,0)}}$. In particular the above subcategories are functorially finite in $\Mod{\Lambda_{(0,0)}}$, closed under isomorphic images, direct sums, direct products, kernels and cokernels.
\begin{proof}
Since $\phi=\psi=0$ it follows that the maps $\theta_1\colon \Lambda_{(0,0)}\lxr A$, $\theta_1\big(\bigl(\begin{smallmatrix}
a & n \\
m & b
\end{smallmatrix}\bigr) \big)=a$, $\theta_2\colon \Lambda_{(0,0)}\lxr B$, $\theta_2\big(\bigl(\begin{smallmatrix}
a & n \\
m & b
\end{smallmatrix}\bigr) \big)=b$, $\theta_3\colon \Lambda_{(0,0)}\lxr \bigl(\begin{smallmatrix}
A & 0 \\
_BM_A & B
\end{smallmatrix}\bigr)$, $\theta_3\big(\bigl(\begin{smallmatrix}
a & n \\
m & b
\end{smallmatrix}\bigr) \big)=\bigl(\begin{smallmatrix}
a & 0 \\
m & b
\end{smallmatrix}\bigr)$ and $\theta_4\colon \Lambda_{(0,0)}\lxr \bigl(\begin{smallmatrix}
A & _AN_B \\
0 & B
\end{smallmatrix}\bigr)$, $\theta_4\big(\bigl(\begin{smallmatrix}
a & n \\
m & b
\end{smallmatrix}\bigr) \big)=\bigl(\begin{smallmatrix}
a & n \\
0 & b
\end{smallmatrix}\bigr)$ are surjective ring homomorphisms. Therefore the above maps are ring epimorphisms, i.e. epimorphisms in the category of rings. Hence from \cite{Gabriel}, see also \cite{Geigle-Lenzing}, we infer that the essential images of the restriction functors $\Mod{A}\lxr \Mod{\Lambda_{(0,0)}}$, $\Mod{B}\lxr \Mod{\Lambda_{(0,0)}}$, $\Mod{\bigl(\begin{smallmatrix}
A & 0 \\
_BM_A & B
\end{smallmatrix}\bigr)}\lxr \Mod{\Lambda_{(0,0)}}$ and $\Mod{\bigl(\begin{smallmatrix}
A & _AN_B \\
0 & B
\end{smallmatrix}\bigr)}\lxr \Mod{\Lambda_{(0,0)}}$ are bireflective subcategories.
Then from \cite{Gabriel}, \cite{Geigle-Lenzing} it follows that the full subcategories $\Mod{A}, \Mod{B}, \Mod{\bigl(\begin{smallmatrix}
A & 0 \\
_BM_A & B
\end{smallmatrix}\bigr)}, \Mod{\bigl(\begin{smallmatrix}
A & _AN_B \\
0 & B
\end{smallmatrix}\bigr)}$ are functorially finite in $\Mod{\Lambda_{(0,0)}}$, closed under isomorphic images, direct sums, direct products, kernels and cokernels. For completeness we describe the approximations of the $\Lambda_{(0,0)}$-modules. Let $(X,Y,f,g)$ be a $\Lambda_{(0,0)}$-module. Recall that we have the full embedding $\mz_A\colon \Mod{A}\lxr \Mod{\Lambda}_{(0,0)}$ sending an $A$-module $X$ to the tuple $(X,0,0,0)$. Then the map $(X,Y,f,g)\lxr (\Coker{g},0,0,0)$ is the unique left $\Image{\mz_A}$-approximation. Let $\F\colon \Mod{\bigl(\begin{smallmatrix}
A & 0 \\
_BM_A & B
\end{smallmatrix}\bigr)}\lxr \Mod{\Lambda_{(0,0)}}$ be the functor defined on the objects $(X,Y,f)\in \Mod{\bigl(\begin{smallmatrix}
A & 0 \\
_BM_A & B
\end{smallmatrix}\bigr)}$ by $\F(X,Y,f)=(X,Y,f,0)$ and given a morphism $(a,b)\colon (X,Y,f)\lxr (X',Y',f')$ in $\Mod{\bigl(\begin{smallmatrix}
A & 0 \\
_BM_A & B
\end{smallmatrix}\bigr)}$ then $\F(a,b)=(a,b)\colon (X,Y,f,0)\lxr (X',Y',f',0)$ is a morphism in $\Mod{\Lambda_{(0,0)}}$. 
We denote by $\U=\Image{\F}=\{(X,Y,f,0) \ | \ (X,Y,f)\in \Mod{\bigl(\begin{smallmatrix}
A & 0 \\
_BM_A & B
\end{smallmatrix}\bigr)}\}$. Let $(X,Y,f,g)$ be a $\Lambda_{(0,0)}$-module. Since $\Phi_Y=0$ we have the following commutative diagram$\colon$
\[
\xymatrix{
 M\otimes_AN\otimes_BY \ar[rrd]_{0=\Phi_Y} \ar[rr]^{M\otimes g} && M\otimes_AX  \ar[d]_{f} \ar[rr]^{M\otimes \pi_X \ \ \ } &&
M\otimes_A \Coker{g} \ar@{-->}[lld]^{h} \ar[r]^{}                  &  0    \\
          &   & Y  &    &  &      }
\]
Consider the object $(\Coker{g},Y,h,0)\in \W$. Then the morphism $(\pi_X,\iden_Y)\colon (X,Y,f,g)\lxr (\Coker{g},Y,h,0)$ is the unique left $\U$-approximation. Similarly we obtain the descriptions of the left approximations from $\Mod{B}$ and $\Mod{\bigl(\begin{smallmatrix}
A & _AN_B \\
0 & B
\end{smallmatrix}\bigr)}
$, and dually we can derive the right approximations. 
\end{proof}
\end{thm}

\begin{rem}
\begin{enumerate}
\item By Proposition \ref{ttf-triple} the categories $\Mod{A}$ and $\Mod{B}$ are \textsf{TTF}-classes. This implies that they are bireflective in $\Mod{\Lambda_{(0,0)}}$, but note that this is clear also from the recollements diagrams of Proposition \ref{recollement}.  

\item For every $\Lambda_{(0,0)}$-module $(X,Y,f,g)$ the assignment $(X,Y,f,g)\mapsto (\Coker{g},Y,h)$ induces a functor $\G\colon \Mod{\Lambda_{(0,0)}}\lxr \Mod{\bigl(\begin{smallmatrix}
A & 0 \\
_BM_A & B
\end{smallmatrix}\bigr)}$ which is the left adjoint of the functor $\F\colon \Mod{\bigl(\begin{smallmatrix}
A & 0 \\
_BM_A & B
\end{smallmatrix}\bigr)}\lxr \Mod{\Lambda_{(0,0)}}$ in the proof of Theorem $4.6$. Similarly we obtain the right adjoint of $\F$ and in the same way we get the adjoints of the full embedding $\Mod{\bigl(\begin{smallmatrix}
A & _AN_B \\
0 & B
\end{smallmatrix}\bigr)}\lxr \Mod{\Lambda_{(0,0)}}$. 
\end{enumerate}
\end{rem}

We continue with the following result on finiteness of subcategories.

\begin{thm}
Let $\Lambda_{(0,0)}$ be a Morita ring. 
\begin{enumerate}
\item Let $\U$ be a covariantly finite subcategory of $\Mod{A}$ such that $\U\subseteq \Ker{\Hom_{A}(N,-)}$ and $\V$ a covariantly finite subcategory of $\Mod{B}$ such that $\V\subseteq \Ker{N\otimes_B-}$. Then the full subcategory
\[
\W=\big\{(X,Y,f,g)\in \Mod{\Lambda_{(0,0)}} \ | \ X\in \U \ \ \text{and} \ \ Y\in \V \big\}
\]
is covariantly finite in $\Mod{\Lambda_{(0,0)}}$.

\item Let $\U$ be a contravariantly finite subcategory of $\Mod{A}$ such that $\U\subseteq \Ker{\Hom_{A}(N,-)}$ and $\V$ a contravariantly finite subcategory of $\Mod{B}$ such that $\V\subseteq \Ker{N\otimes_B-}$. Then the full subcategory
\[
\W=\big\{(X,Y,f,g)\in \Mod{\Lambda_{(0,0)}} \ | \ X\in \U \ \ \text{and} \ \ Y\in \V \big\}
\]
is contravariantly finite in $\Mod{\Lambda_{(0,0)}}$.

\item Let $\U$ be a functorially finite subcategory of $\Mod{A}$ such that $\U\subseteq \Ker{\Hom_{A}(N,-)}$ and $\V$ a functorially finite subcategory of $\Mod{B}$ such that $\V\subseteq \Ker{N\otimes_B-}$. Then the full subcategory
\[
\W=\big\{(X,Y,f,g)\in \Mod{\Lambda_{(0,0)}} \ | \ X\in \U \ \ \text{and} \ \ Y\in \V \big\}
\]
is functorially finite in $\Mod{\Lambda_{(0,0)}}$.
\end{enumerate}
\begin{proof}
(i) Let $(A_1,B_1,f_1,g_1)$ be an arbitrary $\Lambda_{(0,0)}$-module and let $m\colon A_1\lxr X_1$ be a left $\U$-approximation. From the morphisms $M\otimes_Am$ and $f_1$ we have the following pushout diagram$\colon$
\[
\xymatrix{
      M\otimes_A{A_1} \ar[rr]^{M\otimes m}   \ar[d]_{f_1}    &&  M\otimes_AX_1 \ar[d]^{\rho}         \\
 B_1 \ar[rr]^{\theta}  &&  I          }
\]
and let $n\colon I\lxr Y_1$ be a left $\V$-approximation. Then we claim that the map
\[
\xymatrix{
  (A_1,B_1,f_1,g_1) \ar[rr]^{(m,\theta\circ n)}    &&  (X_1,Y_1,\rho\circ n,0) }
\]
is a left $\W$-approximation. First the object $(X_1,Y_1,\rho\circ n,0)\in \W$ since it is a $\Lambda_{(0,0)}$-module with $X_1\in \U$ and $Y_1\in \V$. Also from the above pushout diagram and since $\Hom_A(N\otimes_BB_1,X_1)\simeq \Hom_B(B_1,\Hom_A(N,X_1))=0$ it follows that the following diagrams are commutative$\colon$
\[
\xymatrix{
  M\otimes_A A_1 \ar[d]_{M\otimes m} \ar[r]^{ \ \ \ f_1} & B_1 \ar[d]^{\theta\circ n}     \\
  M\otimes_A X_1    \ar[r]^{ \ \ \ \rho\circ n} & Y_1                  } \ \ \ \ \ \ \  \ \ \  \xymatrix{
  N\otimes_B B_1 \ar[d]_{N\otimes (\theta\circ n)} \ar[r]^{ \ \ g_1} &  A_1 \ar[d]^{m}     \\
  N\otimes_B Y_1    \ar[r]^{ \ \ 0} & X_1                  }
\]
Thus the map $(m,\theta\circ n)$ is a morphism in $\Mod{\Lambda_{(0,0)}}$. Let $(\alpha,\beta)\colon (A_1,B_1,f_1,g_1)\lxr (X,Y,f,g)$ be a morphism in $\Mod{\Lambda_{(0,0)}}$ with $(X,Y,f,g)\in \W$. Since $m\colon A_1\lxr X_1$ is a left $\U$-approximation and $X\in \U$ there exists a map $\gamma\colon X_1\lxr X$ such that $m\circ \gamma=\alpha$. Morover since $f_1\circ \beta=(M\otimes_A a)\circ f$ there exists a map $\mu\colon I\lxr Y$ such that the following diagram$\colon$
\[
\xymatrix{
     M\otimes_AA_1 \ar[rr]^{M\otimes m}   \ar[d]_{f_1} && M\otimes_AX_1 \ar[d]^{\rho} \ar @/^1.5pc/[ddr]^{(M\otimes \gamma)\circ f}  &        \\
 B_1 \ar[rr]^{\theta} \ar @/_1.5pc/[drrr]^{\beta}  &&  I  \ar@{-->}[dr]^{\mu} &  \\
    &&   &  Y   }
\]
is commutative. Then since the map $n\colon I\lxr Y_1$ is a left $\V$-approximation there exists a morphism $\delta\colon Y_1\lxr Y$ such that $n\circ \delta=\mu$. 
This implies that the following diagram is commutative$\colon$
\[
\xymatrix{
   (A_1,B_1,f_1,g_1)  \ar[d]_{(\alpha,\beta)} \ar[rr]^{ (m,\theta\circ n) \ }  &&      (X_1,Y_1,\rho\circ n,0) \ar[dll]^{(\gamma,\delta)}      \\
 (X,Y,f,g)  &&           }
\] 
Since $\rho\circ n\circ \delta=\rho\circ \mu=(M\otimes_A\gamma)\circ f$ and $N\otimes_BY_1=0$ we infer that $(\gamma,\delta)$ is a morphism in $\Mod{\Lambda_{(0,0)}}$ and therefore we have proved our claim. Part (ii) is dual to (i) and (iii) follows from (i) and (ii).
\end{proof}
\end{thm}

\begin{rem}
Note that the converse of Theorem $4.8$ holds, i.e. if $\W$ is contravariantly (resp. covariantly) finite in $\Mod{\Lambda_{(0,0)}}$ then $\U$ is contravariantly (resp. covariantly) finite in $\Mod{A}$ and $\V$ is contravariantly (resp. covariantly) finite in $\Mod{B}$. The proof is not difficult and is left to the reader. 
\end{rem}

If the bimodule $N=0$ then from Theorem $4.8$ we have the following well known result due to Smal{\o}.

\begin{cor}\cite[Theorem $2.1$]{triangular}
Let $\Lambda=\bigl(\begin{smallmatrix}
A & 0 \\
_BM_A & B
\end{smallmatrix}\bigr)$ be a triangular matrix ring, $\U$ a full subcategory of $\Mod{A}$, $\V$ a full subcategory of $\Mod{B}$ and let $\W=\big\{(X,Y,f)\in \Mod{\Lambda} \ | \ X\in \U \ \ \text{and} \ \ Y\in \V \big\}$.
\begin{enumerate}
\item The subcategory $\W$ is covariantly finite in $\Mod{\Lambda}$ if and only if $\U$ is covariantly finite in $\Mod{A}$ and $\V$ is covariantly finite in $\Mod{B}$.
 
\item The subcategory $\W$ is contravariantly finite in $\Mod{\Lambda}$ if and only if $\U$ is contravariantly finite in $\Mod{A}$ and $\V$ is contravariantly finite in $\Mod{B}$.

\item The subcategory $\W$ is functorially finite in $\Mod{\Lambda}$ if and only if $\U$ is functorially finite in $\Mod{A}$ and $\V$ is functorially finite in $\Mod{B}$.
\end{enumerate}
\end{cor}

We continue with the following applications for Artin algebras. For the notion of Auslander-Reiten sequences we refer to \cite{ARS}.

\begin{cor}
Let $\Lambda_{(0,0)}$ be a Morita ring which is an Artin algebra. Then the full subcategories $\smod{A}$ and $\smod{B}$ of $\smod{\Lambda_{(0,0)}}$ have relative Auslander-Reiten sequences in $\smod{\Lambda_{(0,0)}}$.
\begin{proof}
The result follows from \cite{AS:subcategories} and Theorem \ref{bireflective}. 
\end{proof}
\end{cor}

\begin{cor}
Let $\Lambda_{(0,0)}$ be a Morita ring which is an Artin algerba. Let $\U$ be an extension closed functorially finite subcategory of $\smod{A}$ such that $\U\subseteq \Ker{\Hom_{A}(N,-)}$ and $\V$ an extension closed functorially finite subcategory of $\smod{B}$ such that $\V\subseteq \Ker{N\otimes_B-}$. Then the full subcategory $\W=\{(X,Y,f,g)\in \smod{\Lambda_{(0,0)}} \ | \ X\in \U \ \ \text{and} \ \ Y\in \V \}$ has Auslander-Reiten sequences.
\begin{proof}
Since $\U$ and $\V$ are closed under extensions it follows that $\W$ is also closed under extensions. Then the result follows from Theorem $4.8$ and \cite{AS:subcategories}. 
\end{proof}
\end{cor}

The last part in the paper of Smal{\o} \cite{triangular} deals with the full subcategory of modules of finite projective dimension. In particular, if $\Lambda=\bigl(\begin{smallmatrix}
A & 0 \\
_BM_A & B
\end{smallmatrix}\bigr)$ is a triangular matrix Artin algebra, then the category of $\Lambda$-modules of finite projective dimension is contravariantly finite in $\smod{\Lambda}$ if and only if the category of $A$-modules of finite projective dimension is contravariantly finite in $\smod{A}$ and the category of $B$-modules of finite projective dimension is contravariantly finite in $\smod{B}$, see \cite[Proposition $2.3$]{triangular}. This result follows from Corollary $4.10$ and the description of the subcategory of $\Lambda$-modules of finite projective dimension. Recall that a $\Lambda$-module $(X,Y,f)$ is of finite projective dimension if and only if the projective dimensions of $_AX$ and $_BY$ are finite. 

We close this section with the next example which shows that the subcategory of $\Lambda_{(0,0)}$-modules of finite projective dimension cannot be described as in the lower triangular case. This distinguishes our situation from the lower triangular situation.

\begin{exam}  Let $K$ be a field and $K\mathcal Q$ be the path
algebra where $\mathcal Q$  is the quiver
\[\mathcal Q=\xymatrix{
\stackrel{v_1}{\circ}\ar@(ul,dl)_a[]
\ar@/^/^b[r]
&\stackrel{v_2}{\circ}\ar@(dr,ur)_d[]
\ar@/^/_c[l]
}\]
Let $I$ be the ideal generated by $a^2, bc,cb, d^2, ba-db$, and
$cd-ac$ and let $\Lambda=K\mathcal Q/I$.  It is not hard to show
that $\Lambda$ is a selfinjective finite dimensional $K$-algebra. The structure of the indecomposable projective-injective modules look like:
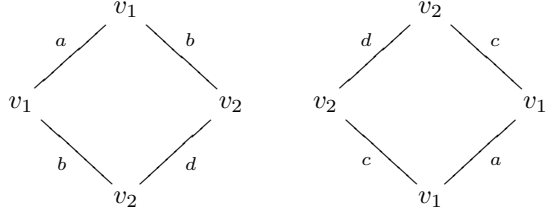
\[
\xymatrix{
&v_1\ar@{-}_a[dl]\ar@{-}^b[dr]\\
v_1\ar@{-}_b[dr]&&v_2\ar@{-}^d[dl]\\
&v_2
}  \quad\quad
\xymatrix{
&v_2\ar@{-}_d[dl]\ar@{-}^c[dr]\\
v_2\ar@{-}_c[dr]&&v_1\ar@{-}^a[dl]\\
&v_1
}\]
 Setting $e=v_1$ and $e'=v_2$,
we view $\Lambda$ as the Morita ring  via 
\[\Lambda_{(\phi,\psi)} = \begin{pmatrix} e\Lambda e & e\Lambda e'\\
e'\Lambda e & e'\Lambda e'
\end{pmatrix}\]
Note that, in this case, $\phi=\psi=0$. One sees that $\Lambda$ is a selfinjective finite dimensional  biserial algebra. Consider
the string module of  the form
\[D=
\xymatrix{
&v_1\ar@{-}[dl]\ar@{-}[dr]&&v_2\ar@{-}[dl]\\
v_1&&v_2
}\]
Viewing $D$ as module over the algebra $\Lambda_{(\phi,\psi)}$,
$D=(X, Y,f,g)$, we see that $X$ is isomorphic to $e\Lambda e$ as 
a left $e\Lambda e$-module, $Y$ is isomorphic to $e'\Lambda e'$ as 
a left $e'\Lambda e'$-module,  and $g=0$.  Thus, we have that
$\pd_{e\Lambda e}{X}<\infty$, $\pd_{e'\Lambda e'}{Y}<\infty$, but
$\pd_{\Lambda }{D}=\infty$ (since $D$ is not a projective $\Lambda$-module
and $\Lambda$ is selfinjective). Finally, letting $R=K[x]/(x^2)$, then it is easy to see that $\Lambda_{(0,0)} \cong \bigl(\begin{smallmatrix}
R & R \\
R & R
\end{smallmatrix}\bigr)$.
\end{exam}

\section{Bounds on the Global Dimension}
Let $\Lambda_{(0,0)}=\bigl(\begin{smallmatrix}
A & _AN_B \\
_BM_A & B
\end{smallmatrix}\bigr)$ be a Morita ring which is an Artin algebra with $\phi=\psi=0$. In this section we show that, under certain restictions on either $M$ or $N$, there is a bound on the global dimension of $\Lambda_{(0,0)}$ in terms of the global dimensions of $A$ and $B$. This is achieved via the notion of tight projective module and tight projective resolution that we introduce in the first subsection. 

\subsection{Tight Resolutions and Upper Bounds} 

Before we begin with some preliminary definitions and results, we give an
example which shows that we will need some restrictions to get a
bound on the global dimension of $\Lambda_{(0,0)}$ in terms of the global
dimensions of $A$ and $B$.

\begin{exam}\label{infgldim}
\rm{
Let $K$ be a field and $\mathcal{Q}$ be the quiver
\[\xymatrix{
\stackrel{v}{\circ}\ar@<1ex>[r]^a&\stackrel{w}{\circ}\ar@<1ex>[l]^b
}\] Let $\Lambda=K\mathcal{Q}/\langle ab,ba\rangle$. Let $P$ (respectively
$Q$) be the projective $\Lambda$-cover of the simple module having $K$
at vertex $v$ (resp. $w$) and $0$ at vertex $w$ (resp. $v$). Then
$\Lambda=P\oplus Q$.   Hence $\Lambda$ is isomorphic to
$\Hom_{\Lambda}(P\oplus Q,P\oplus Q)^{\op}$, which in turn is
isomorphic to the matrix algebra
\[
\begin{pmatrix} \End_{\Lambda}(P)^{\op}& \Hom_{\Lambda}(P,Q)\\
\Hom_{\Lambda}(Q,P) & \End_{\Lambda}(Q)^{\op}\end{pmatrix}
\]
Each entry in this $2\times 2$-matrix is $K$ but the multiplication
of two elements, one of the form $\bigl(\begin{smallmatrix}
0 & 0 \\
\alpha & 0
\end{smallmatrix}\bigr)$ and the other of the form
$\bigl(\begin{smallmatrix}
0 & \beta \\
0 & 0
\end{smallmatrix}\bigr)$, in any
order, is $0$. Thus, as a Morita ring
$\bigl(\begin{smallmatrix}
A & N \\
M & B
\end{smallmatrix}\bigr)$, $A=B=M=N=K$ and $\phi=\psi=0$.
Hence $A$ and $B$ have global dimension $0$ and $M$ and $N$ have
projective dimension $0$ over both $A$ and $B$. But $\Lambda$ has
infinite global dimension. \qed }
\end{exam}

We introduce the following notion which is crucial for our results of this section.

\begin{defn}
If $P=(P_ A,0,0,0)$ is a projective $\Lambda_{(0,0)}$-module for some left
$A$-module $P_A$, then $P$ is called an  
\textsf{$A$-tight projective} $\Lambda_{(0,0)}$-module. We say that a left $\Lambda_{(0,0)}$-module $(X,0,0,0)$ has an \textsf{$A$-tight projective $\Lambda_{(0,0)}$-resolution} if $(X,0,0,0)$ has a projective $\Lambda_{(0,0)}$-resolution in which each
projective $\Lambda_{(0,0)}$-module is $A$-tight.
\end{defn}

Note that if $(P_A,0,0,0)$ is an $A$-tight projective $\Lambda_{(0,0)}$-module then $P_A$ is a projective
$A$-module and $M\otimes_AP_A=0$. Conversely, if $P_A$ is a
projective $A$-module and $M{\otimes_A}P_A=0$, then $(P_A,0,0,0)$ is
an $A$-tight projective $\Lambda_{(0,0)}$-module. It is easy to see the following.
\begin{enumerate}
\item A direct sum of
modules having $A$-tight projective $\Lambda_{(0,0)}$-resolutions also has
an $A$-tight projective $\Lambda_{(0,0)}$-resolution.

\item A summand of an $A$-tight projective $\Lambda_{(0,0)}$-module is again
an $A$-tight projective $\Lambda_{(0,0)}$-module.

\item If $X$ is an $A$-module such that $(X,0,0,0)$ has an $A$-tight projective
$\Lambda_{(0,0)}$-resolution, then $\pd_{A}X=\pd_{\Lambda_{(0,0)}}(X,0,0,0)$.
\end{enumerate}

The next result classifies $\Lambda_{(0,0)}$-modules having $A$-tight
projective $\Lambda_{(0,0)}$-resolutions.

\begin{prop}\label{tightclass} A $\Lambda_{(0,0)}$-module of the form
$(X,0,0,0)$ has an $A$-tight projective $\Lambda_{(0,0)}$-resolution if and
only if $M\otimes_AP=0$, where $P$ is the direct sum of projective
$A$-modules in a minimal projective $A$-resolution of $X$.
\end{prop}
\begin{proof}  
Suppose that $\cdots \lxr P^2 \lxr P^1 \lxr P^0 \lxr X\lxr 0$
is a minimal projective $A$-resolution of $X$. Set $P=\oplus_{n\ge
0}P^n$.  If $M\otimes_A P=0$ then $M\otimes_A P^n=0$, for all $n\ge 0$.
It follows that $(X,0,0,0)$ has an $A$-tight projective
$\Lambda_{(0,0)}$-resolution. On the other hand, if $M\otimes_AP\ne 0$, then there is a smallest
$n\ge 0$, such that $M\otimes_{A} P^n\ne 0$.  It follows that there is a
minimal projective $\Lambda_{(0,0)}$-resolution of $(X,0,0,0)$ that starts
\[
\xymatrix{
  (P^n,M\otimes_AP^n,\iden_{M\otimes P^n},0) \ar[r]^{} & (P^{n-1},0,0,0) \ar[r] & \cdots \ar[r] & (P^0,0,0,0) \ar[r] & (X,0,0,0) \ar[r] & 0    }
\]
Hence the $n+1$-st syzygy is of the form
$(\Omega^{n+1}_A(X),M\otimes_AP^n,\iden_{M}\otimes \kappa,0)$, where $\kappa$ is the monomorphism $\kappa\colon \Omega^{n+1}_A(X)\lxr P^n$. It follows that the next
projective in the above minimal $\Lambda_{(0,0)}$-resolution of $(X,0,0,0)$ is not
$A$-tight and the result follows.
\end{proof}

We use the next result a number of times in what follows.

\begin{lem}\label{lemmapd}
If $P$ is a projective $A$-module such that $(P,0,0,0)$ is not an
$A$-tight projective $\Lambda_{(0,0)}$-module, then$\colon$
\[\pd_{\Lambda_{(0,0)}}(P,0,0,0)=1+\pd_{\Lambda_{(0,0)}}(0,M\otimes_AP,0,0)
\]
If $Q$ is a projective $B$-module such that $(0,Q,0,0)$ is not a
$B$-tight projective $\Lambda_{(0,0)}$-module, then$\colon$
\[\pd_{\Lambda_{(0,0)}}(0,Q,0,0)=1+\pd_{\Lambda_{(0,0)}}(N\otimes_BQ,0,0,0)
\]
\end{lem}
\begin{proof} The first statement follows from the following short
exact sequence of $\Lambda_{(0,0)}$-modules 
\[
\xymatrix{
  0 \ar[r]^{} & (0,M\otimes_AP,0,0) \ar[r] & (P,M\otimes_AP, \iden_{M\otimes_A P},0) \ar[r] & (P,0,0,0) \ar[r] & 0    }
\]
and the proof of the second statement is similar.
\end{proof}

We define \textsf{$B$-tight projective $\Lambda_{(0,0)}$-modules} $(0,Q,0,0)$
in a similar fashion as $A$-tight projective $\Lambda_{(0,0)}$-modules and
also $\Lambda_{(0,0)}$-modules $(0,Y,0,0)$ having \textsf{$B$-tight projective
$\Lambda_{(0,0)}$-resolutions}.

We also have the following result.

\begin{lem}\label{tightlemma} Let $X$ be an $A$-$B$-bimodule such
that $(X,0,0,0)$ has an $A$-tight projective $\Lambda_{(0,0)}$-resolution.
If $Q$ is a projective $B$-module, then $(X\otimes_BQ,0,0,0)$ has an
$A$-tight projective $\Lambda_{(0,0)}$-resolution.
\end{lem}
\begin{proof}  
Since direct sums of modules that have $A$-tight
projective $\Lambda_{(0,0)}$-resolutions are modules having $A$-tight
projective $\Lambda_{(0,0)}$-resolutions, we may assume that $Q=Be$, for some
primitive idempotent $e$ in $B$. Since $X\otimes_B Be\simeq Xe$ is a
summand of $X$, the result follows.
\end{proof}

The next lemma is a useful tool in what follows.

\begin{lem}\label{pdlemma} Let $X$ be an $A$-module and $Y$ be a $B$-module. Then$\colon$
\begin{enumerate}
\item $\pd_{\Lambda_{(0,0)}}(X,0,0,0) \ \le \ 1 +
\maxx\lbrace\pd_{\Lambda_{(0,0)}}(\Omega^1_A(X),0,0,0),\pd_{\Lambda_{(0,0)}}(0,M,0,0)\rbrace
$

\item $\pd_{\Lambda_{(0,0)}}(0,Y,0,0) \ \le \ 1 +
\maxx\lbrace\pd_{\Lambda_{(0,0)}}(0,\Omega^1_B(Y),0,0),\pd_{\Lambda_{(0,0)}}(N,0,0,0)\rbrace
$
\end{enumerate}
\end{lem}
\begin{proof}  
We only prove (i) since the proof of (ii) is similar. Let $\alpha\colon P\lxr X$ be a projective $A$-cover of $X$ with kernel $\Omega^1_A(X)$. Then we have a short exact sequence
\[
\xymatrix{
  0 \ar[r]^{} & (\Omega^1_A(X),M\otimes_{A} P,M\otimes k,0) \ar[rr]^{ \ \ \ (k,\iden_{M\otimes P})} && (P,M\otimes_{A} P,\iden_{M\otimes P},0) \ar[r]^{ \ \ \ \ \ \ \ (\alpha,0)} & (X,0,0,0) \ar[r] & 0    }
\]
in which the middle term is a projective $\Lambda_{(0,0)}$-module. Therefore it
follows that $\pd_{\Lambda_{(0,0)}}(X,0,0,0)\le 1 +
\pd_{\Lambda_{(0,0)}}(\Omega^1_A(X),M\otimes_{A} P,M\otimes k,0)$. Next
we note that we have a short exact sequence
\[
\xymatrix{
0 \ar[r] & (0,M\otimes_{A} P,0,0)\ar[r] & (\Omega^1_A(X),M\otimes_{A} P,M\otimes
k,0)\ar[r] & (\Omega^1_A(X),0,0,0)\ar[r] & 0}
\] 
We observe that $M\otimes_AP$ is direct sum of summands of $M$. Hence we infer that
$\pd_{\Lambda_{(0,0)}}(0,M\otimes_AP,0,0) \le \pd_{\Lambda_{(0,0)}}(0,M,0,0)$ and the
result now follows.
\end{proof}

We get an immediate consequence of the previous lemma.

\begin{cor}\label{prepd} Let $X$ be an $A$-module
and $Y$ be a $B$-module. Then$\colon$
\begin{enumerate}
\item $\pd_{\Lambda_{(0,0)}}(X,0,0,0) \ \le \ \pd_AX+1 +
\pd_{\Lambda_{(0,0)}}(0,M,0,0) $

\item $\pd_{\Lambda_{(0,0)}}(0,Y,0,0) \ \le \ \pd_BY+1 +
\pd_{\Lambda_{(0,0)}}(N,0,0,0) $
\end{enumerate}
\end{cor}
\begin{proof} We only prove (i). If $\pd_AX$ is not finite, then the result
follows. Assume that $\pd_AX=n<\infty$. If we apply Lemma $5.6$ first to
$(X,0,0,0)$ and then to $(\Omega^1_A(X),0,0,0)$, we get
\[
\pd_{\Lambda_{(0,0)}}(X,0,0,0)\le 2 +
\maxx\lbrace\pd_{\Lambda_{(0,0)}}(\Omega^2_A(X),0,0,0),\pd_{\Lambda_{(0,0)}}(0,M,0,0)\rbrace
\] 
Continuing in this fashion, we get
\[
\pd_{\Lambda_{(0,0)}}(X,0,0,0)\le n +
\maxx\lbrace\pd_{\Lambda_{(0,0)}}(\Omega^n_A(X),0,0,0),\pd_{\Lambda_{(0,0)}}(0,M,0,0)\rbrace
\] 
By assumption the n-th syzygy $\Omega^n_A(X)$ is a projective $A$-module.
Applying Lemma $5.6$ to $(\Omega^n_A(X),0,0,0)$, we obtain the desired
result.
\end{proof}

We are now in a position to state our first set of results. For simplicity we write that a left $A$-module $X$ has an $A$-tight projective $\Lambda_{(0,0)}$-resolution meaning that the object $(X,0,0,0)$, as a left $\Lambda_{(0,0)}$-module, has an $A$-tight projective $\Lambda_{(0,0)}$-resolution. We make the same agreement for left $B$-modules having $B$-tight projective $\Lambda_{(0,0)}$-resolution. 

\begin{prop}\label{MNtight} 
Let $\Lambda_{(0,0)}=\bigl(\begin{smallmatrix}
A & _AN_B \\
_BM_A & B
\end{smallmatrix}\bigr)$ be a Morita ring which is an Artin algebra and let $X$ be an $A$-module and $Y$ be a $B$-module. If $M$ 
has a $B$-tight projective $\Lambda_{(0,0)}$-resolution, then
\[
\pd_{\Lambda_{(0,0)}}(X,0,0,0) \ \le \ \pd_AX+1 + \pd_BM
\]
If $N$ has an $A$-tight projective $\Lambda_{(0,0)}$-resolution, then
\[
\pd_{\Lambda_{(0,0)}}(0,Y,0,0) \ \le \ \pd_BY+1 + \pd_AN
\]
\end{prop}
\begin{proof} The result follows from Corollary \ref{prepd} and
the fact that $(0,M,0,0)$ having a $B$-tight projective resolution
implies that $\pd_{\Lambda_{(0,0)}}(0,M,0,0)=\pd_BM$. Similarly we obtain the second inequality.
\end{proof}

\begin{thm}\label{MNtightthm}
Let $\Lambda_{(0,0)}=\bigl(\begin{smallmatrix}
A & _AN_B \\
_BM_A & B
\end{smallmatrix}\bigr)$
be a Morita ring which is an Artin algebra and suppose that $M$ has a $B$-tight projective $\Lambda_{(0,0)}$-resolution and $N$ has an $A$-tight projective
$\Lambda_{(0,0)}$-resolution. Then$\colon$
\[
\gld\Lambda_{(0,0)} \ \le \ \gld{A}+\gld{B}+1
\]
\end{thm}
\begin{proof} 
Since $\phi=\psi=0$ it follows from Proposition \ref{simples} that 
the simple $\Lambda_{(0,0)}$-modules are of the form $(S,0,0,0)$, where $S$ is a simple $A$-module or of the form $(0,T,0,0)$, where $T$ is a simple $B$-module. Now
\[
\gld\Lambda_{(0,0)} \le \maxx\lbrace
\pd_{\Lambda_{(0,0)}}(S,0,0,0),\pd_{\Lambda_{(0,0)}}(0,T,0,0) \mid S\colon \text{simple
} A\text{-module}, \ T\colon \text{simple} \ B\text{-module}\rbrace
\]
By Proposition \ref{MNtight} we have $\pd_{\Lambda_{(0,0)}}(S,0,0,0)\le
\pd_AS+\pd_BM+1$. Thus, $\pd_{\Lambda_{(0,0)}}(S,0,0,0)\le \gld{A}
+\gld{B}+1$. Similarly, we infer that $\pd_{\Lambda_{(0,0)}}(0,T,0,0)\le \gld{A}
+\gld{B}+1$ and then the result follows.
\end{proof}

We provide two examples, the first of which shows that the
inequality of Theorem \ref{MNtightthm} is sharp and the second shows that the inequality can be proper.

\begin{exam}\label{sharp}{\rm Let $K$ be a field and $\mathcal{Q}$ the
quiver
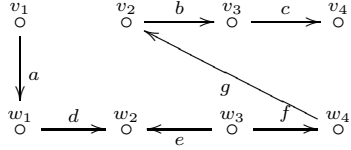
\[
\xymatrix{
\stackrel{v_1}{\circ}\ar[d]^a&\stackrel{v_2}{\circ}\ar[r]^b&\stackrel{v_3}{\circ}\ar[r]^c&\stackrel{v_4}{\circ}\\
\stackrel{w_1}{\circ}\ar[r]^d&\stackrel{w_2}{\circ}&\stackrel{w_3}{\circ}\ar[l]^e\ar[r]^f&\stackrel{w_4}{\circ}\ar[ull]^g
}\] Let $I$ be the ideal in $K\mathcal{Q}$ generated by all paths of length 2
and let $\Lambda= K\mathcal{Q}/I$.  We see the global dimension of $\Lambda$
is 4.  Now set $\epsilon_1=v_1+v_2+v_3+v_4$ and
$\epsilon_2=w_1+w_2+w_3+w_4$.  View $\Lambda$ as the Morita ring
\[\begin{pmatrix} \epsilon_1\Lambda\epsilon_1 &
\epsilon_1\Lambda\epsilon_2\\
\epsilon_2\Lambda\epsilon_1 & \epsilon_2\Lambda\epsilon_2
\end{pmatrix}
\]
The global dimension of $\epsilon_1\Lambda\epsilon_1$ is 2 and the
global dimension of $\epsilon_2\Lambda\epsilon_2$ is 1. Thus
\[\gld{\Lambda}=\gld{\epsilon_1\Lambda\epsilon_1}
+\gld{\epsilon_2\Lambda\epsilon_2}+1
\]
Now $M=\epsilon_2\Lambda\epsilon_1$, which, as a left
$\epsilon_2\Lambda\epsilon_2$-module is isomorphic to the simple module
at $w_1$. We see that $N=\epsilon_1\Lambda\epsilon_2$, which, as a
left $\epsilon_1\Lambda\epsilon_1$-module is isomorphic to the
simple module at $v_2$.  The reader may check that $(0,M,0,0)$ and
$(N,0,0,0)$ have tight projective $\Lambda$-resolutions.  We note
that $\phi$ and $\psi$ are both $0$ for this example.\qed
}
\end{exam}

\begin{exam}{\rm Let $K$ be a field and $\mathcal{Q}$ the quiver
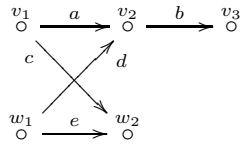
\[
\xymatrix{
\stackrel{v_1}{\circ}\ar[r]^a\ar[rd]_<<c&\stackrel{v_2}{\circ}\ar[r]^b&\stackrel{v_3}{\circ}\\
\stackrel{w_1}{\circ}\ar[ur]_>>d\ar[r]^e&\stackrel{w_2}{\circ}
}
\]
We again take $I$ to be the ideal generated by all paths of length $2$ and
set $\Lambda= K\mathcal{Q}/I$.  Now set $\epsilon_1=v_1+v_2+v_3$ and
$\epsilon_2=w_1+w_2$.  View $\Lambda$ as the Morita ring
\[
\begin{pmatrix} \epsilon_1\Lambda\epsilon_1 &
\epsilon_1\Lambda\epsilon_2\\
\epsilon_2\Lambda\epsilon_1 & \epsilon_2\Lambda\epsilon_2
\end{pmatrix}
\]
The reader may check that the hypotheses of Theorem \ref{MNtightthm} are
satisfied. But the global dimension of $ \Lambda$ is $2$ while the
global dimension of $ \epsilon_1\Lambda\epsilon_1$ is $2$ and the
global dimension of $ \epsilon_2\Lambda\epsilon_2$ is $1$.\qed
}
\end{exam}

We now turn to the case where either $(0,M,0,0)$ or $(N,0,0,0)$ does
not have a tight projective $\Lambda_{(0,0)}$-resolution. If $M$ is not a
projective $B$-module then a projective cover of $(0,M,0,0)$ is of
the form $(0,\beta)\colon (N\otimes_BQ,Q,0,\iden_{N\otimes Q} )\lxr
(0,M,0,0)$, where $\beta\colon Q\lxr M$ is projective $B$-cover of
$M$. In particular, $N\otimes_BQ$ is a sum of summands of $N$ over
which we have little control. The next bound results will require
that $M$, as a left $B$-module, is projective and $N$, as a left
$A$-module is projective.

We state a preliminary lemma.

\begin{lem}\label{tensorproj} Suppose $M$ is a $B$-$A$-bimodule which
is projective as a left $B$-module and $N$ is an $A$-$B$-bimodule
which is projective as a left $A$-module. Then
\begin{enumerate}
\item $M\otimes_A(N\otimes_BM)^{\stackrel{s}{\otimes_A}}$ is a
projective left $B$-module, for all $s\ge 0$.
\item $(N\otimes_BM)^{\stackrel{s}{\otimes_A}}$ is a
projective left $B$-module, for all $s\ge 1$.
\item $N\otimes_B(M\otimes_AN)^{\stackrel{t}{\otimes_B}}$ is a
projective left $A$-module, for all $t\ge 0$.
\item $(M\otimes_AN)^{\stackrel{t}{\otimes_B}}$ is a
projective left $A$-module, for all $t\ge 1$.

\end{enumerate}
\end{lem}
\begin{proof}
Suppose $P$ is a projective left $A$-module. Then $P$ is isomorphic
to a direct sum of indecomposable projective $A$-modules of the form
$Ae$, where $e$ is a primitive idempotent in $A$. It follows that
$M\otimes_AP$ is a direct sum of modules of the form $Me$. Since $M$
is assumed to be a projective left $B$-module, $Me$ is a projective
left $B$-module and, hence, $M\otimes_A P$ is a projective left
$B$-module. Similarly, if $Q$ is a projective left $B$-module then
$N\otimes_B Q$ is also a projective left $A$-module. The result now
follows by induction.
\end{proof}

The next result concerns tight projective $\Lambda_{(0,0)}$-modules.

\begin{lem}\label{tensortight} Suppose $M$ is a $B$-$A$-bimodule which
is projective as a left $B$-module and $N$ is an $A$-$B$-bimodule
which is projective as a left $A$-module. 
\begin{enumerate}
\item If $(0,M\otimes_A(N\otimes_BM)^{\stackrel{s}{\otimes_A}},0,0)$ is a
$B$-tight projective $\Lambda_{(0,0)}$-module, for some $s\ge 0$, then
$(0,(M\otimes_AN)^{\stackrel{s+1}\otimes_B},0,0)$ also is a
$B$-tight projective $\Lambda_{(0,0)}$-module.

\item If $((N\otimes_BM)^{\stackrel{s}{\otimes_A}},0,0,0)$ is an
$A$-tight projective $\Lambda_{(0,0)}$-module, for some $s\ge 0$, then
$(N\otimes_B(M\otimes_AN)^{\stackrel{s}\otimes_B},0,0,0)$ also is an
$A$-tight projective $\Lambda_{(0,0)}$-module.

\item If $((N\otimes_BM)^{\stackrel{s}{\otimes_A}},0,0,0)$ is an
$A$-tight projective $\Lambda_{(0,0)}$-module, for some $s\ge 0$, then
$(N\otimes_B(M\otimes_AN)^{\stackrel{s+1}\otimes_B},0,0,0)$ also is
an $A$-tight projective $\Lambda_{(0,0)}$-module.

\item If $(0,(M\otimes_AN)^{\stackrel{s}{\otimes_B}},0,0)$ is a
$B$-tight projective $\Lambda_{(0,0)}$-module, for some $s\ge 0$, then
$(0,M\otimes_A(N\otimes_BM)^{\stackrel{s}\otimes_A},0,0)$ also is a
$B$-tight projective $\Lambda_{(0,0)}$-module.
\end{enumerate}
\end{lem}
\begin{proof} 
We only prove part (i) since the proofs of the other parts are similar. Assume that
$(0,M\otimes_A(N\otimes_BM)^{\stackrel{s}{\otimes_A}},0,0)$ is a
$B$-tight projective $\Lambda_{(0,0)}$-module, for some $s\ge 0$.  Then
tensoring $M\otimes_A(N\otimes_BM)^{\stackrel{s}{\otimes_A}}$ on the
right by $\otimes_AN$, we obtain
$(M\otimes_AN)^{\stackrel{s+1}\otimes_B}$. The assumption that $N$
is a projective left $A$-module implies that
$(M\otimes_AN)^{\stackrel{s+1}\otimes_B}$ is a direct sum of
summands of $M\otimes_A(N\otimes_BM)^{\stackrel{s}{\otimes_A}}$. The
result now follows.
\end{proof}

We are now in a position to state the second result on bounding the
global dimension of $\Lambda_{(0,0)}$.

\begin{thm}\label{tensorthm}
Let $\Lambda_{(0,0)}=\bigl(\begin{smallmatrix}
A & _AN_B \\
_BM_A & B
\end{smallmatrix}\bigr)$ be a Morita ring which is an Artin algebra. Suppose that the global dimensions of $A$ and $B$ are finite, $M$ is a projective left $B$-module, and
$N$ is a projective left $A$-module.
\begin{enumerate}
\item If $((N\otimes_BM)^{\stackrel{s}{\otimes_A}},0,0,0)$ is an $A$-tight
projective $\Lambda_{(0,0)}$-module, for some $s\ge 1$, then$\colon$
\[\gld\Lambda_{(0,0)} \ \le \ \maxx\lbrace \gld{A}+2s, \gld{B}+2s+1\rbrace
\]
\item If $(0,M\otimes_A(N\otimes_BM)^{\stackrel{s}{\otimes_A}},0,0)$ is a $B$-tight
projective $\Lambda_{(0,0)}$-module, for some $s\ge 0$, then$\colon$
\[\gld\Lambda_{(0,0)} \ \le \ \maxx\lbrace \gld{A}+2s+1, \gld{B}+2(s+1)\rbrace
\]
\item If $(N\otimes_B(M\otimes_AN)^{\stackrel{s}{\otimes_B}},0,0,0)$ is an $A$-tight
projective $\Lambda_{(0,0)}$-module, for some $s\ge 0$, then$\colon$
\[\gld\Lambda_{(0,0)} \ \le \ \maxx\lbrace \gld{A}+2(s+1), \gld{B}+2s+1\rbrace
\]
\item If $(0,(M\otimes_AN)^{\stackrel{s}{\otimes_B}},0,0)$ is a $B$-tight
projective $\Lambda_{(0,0)}$-module, for some $s\ge 1$, then$\colon$
\[\gld\Lambda_{(0,0)} \ \le \ \maxx\lbrace \gld{A}+2s+1, \gld{B}+2s\rbrace
\]
\item If $((N\otimes_BM)^{\stackrel{s}{\otimes_A}},0,0,0)$ is an $A$-tight
projective $\Lambda_{(0,0)}$-module and if \linebreak
$(0,(M\otimes_AN)^{\stackrel{s}{\otimes_B}},0,0)$ is an $B$-tight
projective $\Lambda_{(0,0)}$-module, for some $s\ge 1$, then$\colon$
\[\gld\Lambda_{(0,0)} \ \le \ \maxx\lbrace \gld{A}+2s, \gld{B}+2s\rbrace
\]
\item If $(N\otimes_B(M\otimes_AN)^{\stackrel{s}{\otimes_B}},0,0,0)$ is an $A$-tight
projective $\Lambda_{(0,0)}$-module, and if \linebreak
$(0,M\otimes_A(N\otimes_BM)^{\stackrel{s}{\otimes_A}},0,0)$ is a
$B$-tight projective $\Lambda_{(0,0)}$-module, for some $s\ge 0$, then$\colon$
\[\gld\Lambda_{(0,0)} \ \le \ \maxx\lbrace \gld{A}+2s+1, \gld{B}+2s+1\rbrace
\]
\end{enumerate}
\end{thm}
\begin{proof} 
Let $\gld{A}=d<\infty$ and $\gld{B}=e<\infty$.
We only prove part (i) with the remaining parts having similar
proofs. We assume that
$((N\otimes_BM)^{\stackrel{s}{\otimes_A}},0,0,0)$ is an $A$-tight
projective $\Lambda_{(0,0)}$-module, for some $s\ge 1$. First let $S$ be a
simple $A$-module.  By Lemma \ref{pdlemma} we have
$\pd_{\Lambda_{(0,0)}}(S,0,0,0)\le 1 +
\maxx\lbrace\pd_{\Lambda_{(0,0)}}(\Omega^1_A(S),0,0,0),\pd_{\Lambda_{(0,0)}}(0,M,0,0)\rbrace$. By applying Lemma \ref{pdlemma} again, this time to
$(\Omega^1_A(S),0,0,0)$, we get 
\[
\pd_{\Lambda_{(0,0)}}(S,0,0,0)\le 2 +
\maxx\lbrace\pd_{\Lambda_{(0,0)}}(\Omega^2_A(S),0,0,0),\pd_{\Lambda_{(0,0)}}(0,M,0,0)\rbrace
\]  
Continuing in this fashion, we get 
\[
\pd_{\Lambda_{(0,0)}}(S,0,0,0)\le d+
\maxx\lbrace\pd_{\Lambda_{(0,0)}}(\Omega^d_A(S),0,0,0),\pd_{\Lambda_{(0,0)}}(0,M,0,0)\rbrace
\]
Now $\Omega^d_A(S)$ is a projective $A$-module, so the next time we
apply Lemma \ref{pdlemma}, we obtain 
\[
\pd_{\Lambda_{(0,0)}}(S,0,0,0)\le d+1 + \pd_{\Lambda_{(0,0)}}(0,M,0,0)
\]
If $(0,M,0,0)$ is a $B$-tight projective $\Lambda$-module, then we are done.
Suppose that $(0,M,0,0)$ is not a $B$-tight projective module.
Since $M$ is a projective $B$-module, by Lemma \ref{lemmapd} it follows that $\pd_{\Lambda_{(0,0)}}(0,M,0,0)\le 1 +
\pd_{\Lambda_{(0,0)}}(N\otimes_BM,0,0,0)$. 
If $(N\otimes_BM,0,0,0)$ is an $A$-tight projective $\Lambda$-module, we are done.
Suppose that  $(N\otimes_BM,0,0,0)$ is not an $A$-tight projective $\Lambda$-module.
Since $N\otimes_BM$ is a
projective $A$-module by Lemma \ref{tensorproj}, we see again by Lemma \ref{lemmapd} that $\pd_{\Lambda_{(0,0)}}(0,M,0,0)\le 2 +
\pd_{\Lambda_{(0,0)}}(0,M\otimes_AN\otimes_BM,0,0)$. Continuing in this
fashion, we obtain
\[
\pd_{\Lambda_{(0,0)}}(0,M,0,0)\le 2s-1 +
\pd_{\Lambda_{(0,0)}}((N\otimes_BM)^{\stackrel{s}{\otimes_A}},0,0,0)
\]
By assumption, $((N\otimes_BM)^{\stackrel{s}{\otimes_A}},0,0,0)$ is
an $A$-tight projective $\Lambda_{(0,0)}$-module. Hence, we see that
\[
 \pd_{\Lambda_{(0,0)}}(S,0,0,0)\le d+2s=\gld{A}+2s \eqno(*)
\]
Now let $T$ be a simple $B$-module. By Lemma \ref{pdlemma} we have
\[
\pd_{\Lambda_{(0,0)}}(0,T,0,0)\le 1 +
\maxx\lbrace\pd_{\Lambda_{(0,0)}}(0,\Omega^1_B(T),0,0),\pd_{\Lambda_{(0,0)}}(N,0,0,0)\rbrace
\]
Continuing in a similar fashion to the first part of the proof, we obtain
\[
\pd_{\Lambda_{(0,0)}}(T,0,0,0)\le e+
\maxx\lbrace\pd_{\Lambda_{(0,0)}}(0,\Omega^e_B(T),0,0),\pd_{\Lambda_{(0,0)}}(N,0,0,0)\rbrace
\]
Now $N$ is a projective $A$-module, and we see by Lemma \ref{lemmapd} that $\pd_{\Lambda_{(0,0)}}(N,0,0,0)\le 1 +
\pd_{\Lambda_{(0,0)}}(0,M\otimes_AN,0,0)$. Again, following similar
arguments to the first part of the proof, we obtain
\[
\pd_{\Lambda_{(0,0)}}(N,0,0,0)\le 2s +
\pd_{\Lambda_{(0,0)}}(N\otimes_B(M\otimes_AN)^{\stackrel{s}{\otimes_B}},0,0,0)
\]
By Lemma \ref{tensortight},
$(N\otimes_B(M\otimes_AN)^{\stackrel{s}{\otimes_B}},0,0,0)$ is an
$A$-tight projective $\Lambda_{(0,0)}$-module, and we have
\[
\pd_{\Lambda_{(0,0)}}(0,T,0,0)\le e+ 2s+1=\gld{B}+2s+1 \eqno(**)
\]
Since $\Lambda_{(0,0)}$ is an Artin algebra and since from Proposition \ref{simples} a simple $\Lambda_{(0,0)}$-module is isomorphic to either a module of the form
$(S,0,0,0)$ or $(0,T,0,0)$, for some simple $A$-module $S$ or some
simple $B$-module $T$, part (1) follows from $(*)$ and $(**)$.
\end{proof}

We conclude this section with an example showing that the bounds in
the above theorem are sharp.

\begin{exam}{\rm
Let $K$ be a field and let $\mathcal{Q}$ be the quiver
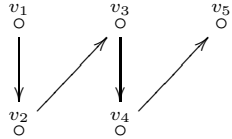
\[
\xymatrix{ \stackrel{v_1}{\circ}\ar[d]&\stackrel{v_3}{\circ}\ar[d]&
\stackrel{v_5}{\circ}\\
\stackrel{v_2}{\circ}\ar[ur]&\stackrel{v_4}{\circ}\ar[ur]
 }\]
Let $\Lambda$ be the quotient $K\mathcal{Q}/I$, where $I$ is the ideal
generated by all paths of length $2$.  Let $\epsilon_1=v_1+v_3+v_5$
and $\epsilon_2=v_2+v_4$.  View $\Lambda$ as the Morita ring
\[
\begin{pmatrix} \epsilon_1\Lambda\epsilon_1 &
\epsilon_1\Lambda\epsilon_2\\
\epsilon_2\Lambda\epsilon_1 & \epsilon_2\Lambda\epsilon_2
\end{pmatrix}
\]
Using the notation $_iK_j$ to denote the simple $\Lambda$-module,
which on the left is isomorphic to the simple $\Lambda$-module at
vertex $v_i$, and on the right is isomorphic to the simple
$\Lambda$-module at vertex $v_j$, we see that
\[
M= \epsilon_2\Lambda\epsilon_1= _4K_3\oplus _2K_1, \text{ and } N=
\epsilon_1\Lambda\epsilon_2= _5K_4\oplus _3K_2.
\]
Now the global dimensions of $A=\epsilon_1\Lambda\epsilon_1$ and $B=
\epsilon_2\Lambda\epsilon_2$ are both $0$.  Clearly, $M$ is a
projective left $B$-module and $N$ is a projective left $A$-module.
We see that $N\otimes_B(M\otimes_AN)$ is isomorphic to $_5K_2$.
Moreover, $(N\otimes_B(M\otimes_AN),0,0,0)$ is an $A$-tight
projective $\Lambda$-module. Thus, we can apply part (3) of Theorem
\ref{tensorthm} with $s=1$ to get $\gld\Lambda\le 4$. But the
global dimension of $\Lambda$ is 4, and we have shown that the
inequality in part (3) is sharp.  This example can be adjusted to
get that all the inequalities are sharp. \qed
}\end{exam}

\subsection{Some Lower Bounds} 
In this subsection we provide some lower bounds for the global dimension of a Morita ring. 

\begin{lem}
Let $\Lambda_{(\phi,\psi)}$ be a Morita ring.
\begin{enumerate}
\item If the bimodule $_BM_A$ is flat as a right $A$-module then
$\pd_{A}{X}=\pd_{\Lambda_{(\phi,\psi)}}{\mt_{A}(X)}$.

\item If the bimodule $_AN_B$ is flat as a right $B$-module then
$\pd_{B}{Y}=\pd_{\Lambda_{(\phi,\psi)}}{\mt_{B}(Y)}$.

\end{enumerate}
\begin{proof}
Suppose that the bimodule $_BM_A$ is flat as a right $A$-module. Then the functor $M\otimes_A-\colon \Mod{A}\lxr \Mod{B}$ is exact and therefore the
functor $\mt_{A}\colon \Mod{A}\lxr \Mod{\Lambda_{(\phi,\psi)}}$ is exact. Let $X$ be a $A$-module with $\pd_{A}{X}=n$ and let $0 \lxr P_n \lxr \cdots \lxr P_0 \lxr X \lxr 0$ be the projective resolution of $X$. Then if we apply the exact functor $\mt_A$ we get that $\pd_{\Lambda_{(\phi,\psi)}}{\mt_{A}(X)}\leq n=\pd_{A}{X}$ since $\mt_{A}$ preserves projectives. Conversely suppose that $\pd_{\Lambda_{(\phi,\psi)}}{\mt_{A}(X)}=m<\infty$. Let $0 \lxr K_0 \lxr P_0 \lxr X \lxr 0 $ be an exact sequence with $P_0\in \Proj{A}$ and $K_0=\Ker{a_0}$. Since
$\mt_{A}$ is exact the sequence $0 \lxr \mt_{A}(K_0) \lxr \mt_{A}(P_0) \lxr \mt_{A}(X) \lxr 0 $ is exact. Now we continue with the same procedure. This means that
we take an epimorphism $a_1\colon P_1\lxr K_0$ with $P_1$ a projective $A$-module, $K_1=\Ker{a_1}$ and then we apply the functor $\mt_{A}$. After $m$-steps we obtain the exact sequence$\colon 0 \lxr \mt_{A}(K_{m-1}) \lxr \mt_{A}(P_{m-1}) \lxr \cdots \lxr \mt_{A}(P_0) \lxr \mt_{A}(X) \lxr 0$ where $\mt_{A}(K_{m-1})$ is projective since $\pd_{\Lambda_{(\phi,\psi)}}{\mt_{A}(X)}=m$. Then if we apply the functor $\mU_{A}$ we get the exact sequence$\colon 0 \lxr K_{m-1} \lxr P_{m-1} \lxr \cdots \lxr P_0 \lxr X \lxr 0$ and we claim that $\Omega^m(X)=K_{m-1}$ is projective in $\Mod{A}$. But this is straightforward since $\mt_{A}(K_{m-1})$ is projective. Thus we have $\pd_{A}{X}\leq m=\pd_{\Lambda_{(\phi,\psi)}}{\mt_{A}(X)}$. We infer that $\pd_A{X}=\pd_{\Lambda_{(\phi,\psi)}}{\mt_{A}(X)}$ and similarly we prove that $\pd_{B}{Y}=\pd{_{\Lambda_{(\phi,\psi)}}{\mt_{B}(Y)}}$ when the functor $N\otimes_B-\colon\Mod{B}\lxr \Mod{A}$ is exact.
\end{proof}
\end{lem}

As a consequence of the above result we have the following lower bound.

\begin{prop}\cite[Lemma $1.2$]{Loustaunau: homol}
Let $\Lambda_{(\phi,\psi)}$ be a Morita ring and suppose that $M_A$ is a flat right $A$-module and $N_B$ is a flat right $B$-module. Then$\colon$
\[
\gld{\Lambda_{(\phi,\psi)}} \ \geq \ \maxx{\{\gld{A},\gld{B}\}}
\]
\end{prop}

\subsection{Comparing Tight Resolutions} 
In this subsection we discuss the assumption of Theorem \ref{MNtightthm} about tight resolutions. Our aim is to compare our result with some well known bounds for the global dimension of trivial extensions rings. 
 
Let $\Lambda_{(0,0)}$ be a Morita ring regarded as an Artin algebra. Then from Proposition \ref{trivial} we have the isomorphism of rings $\Lambda_{(0,0)}\simeq (A\times B)\ltimes M\oplus N$, where $(A\times B)\ltimes M\oplus N$ is the trivial
extension ring of $A\times B$ by the $(A\times B)$-$(A\times
B)$-bimodule $M\oplus N$. Then the module category $\smod{\Lambda_{(0,0)}}$ is equivalent to the trivial extension of abelian categories $(\smod{A}\times \smod{B})\ltimes H$, see \cite{FGR}, where $H$ is the endofunctor
\[
H\colon \smod{A}\times \smod{B}\lxr \smod{A}\times \smod{B}, \ H(X,Y)=(N\otimes_BY,M\otimes_AX)
\]
Suppose that $_AN_B$ has an $A$-tight projective $\Lambda_{(0,0)}$-resolution and $_BM_A$ has a $B$-tight projective $\Lambda_{(0,0)}$-resolution. This implies that we have projective resolutions $\cdots \lxr {_AP_1} \lxr {_AP_0} \lxr {_AN} \lxr 0$ and $\cdots \lxr {_BQ_1} \lxr {_BQ_0} \lxr {_BM} \lxr 0$  
such that $M\otimes_AP_i=0$ and $N\otimes_BQ_i=0$. If we aply the functor $M\otimes_A-$ to the projective resolution of $N$ we get that $M\otimes_AN=0$. Similarly if we apply the functor $N\otimes_B-$ to the projective resolution of $M$ we obtain that $N\otimes_BM=0$. Also we derive that $\Tor_i^A(M,N)=0$ and $\Tor_i^B(N,M)=0$ for every $i\geq 0$. Since $M\otimes_AN=0$ if and only if $M\otimes_AN\otimes_B-=0$ and $N\otimes_BM=0$ if and only if $N\otimes_BM\otimes_A-=0$ it follows that $H^2=0$. From Corollary $7.6$ of \cite{Bel} it follows that if the left derived functor $\mL_iH^j(H(P,Q))=0$ for every $i,j\geq 1$ and $P\in \proj{A}, Q\in \proj{B}$, then
\[  
\gld{\Lambda_{(0,0)}}\le c(H)+2\cdot \max\{\gld{A},\gld{B} \} \eqno(*)
\]
where $c(H)=\min\{\kappa\in \mathbb N:H^{\kappa+1}=0 \}$ is the nilpotency class of $H$. From the projective resolutions of $N$ and $M$ we have the following projective resolution of $H(A,B)\colon$
\[
\xymatrix@C=0.5cm{
  \cdots \ar[rr] && ({_AP_1},{_BQ_1}) \ar[rr]^{} && ({_AP_0},{_BQ_0}) \ar[rr]^{} && H(A,B) \ar[rr]^{} && 0   }
\]
in $\smod{{A\times B}}$. Hence if we apply the functor $H$ to the above exact sequence we obtain the zero complex. We infer that $\mL_iH^j(H(P,Q))=0$ for every $i, j\geq 1$ and $P\in \proj{A}, Q\in \proj{B}$. Hence the assumption of Corollary $7.6$ of \cite{Bel} is satisfied and so we have the bound of the relation $(*)$. In particular we obtain that $\gld{\Lambda_{(0,0}}\le 1+2\cdot \maxx\{\gld{A},\gld{B} \}$ since $H^2=0$. But the bound of Theorem \ref{MNtightthm} is $\gld{\Lambda_{(0,0)}}\le \gld{A}+\gld{B}+1$ which is better than the above bound but the assumption of Corollary $7.6$ of \cite{Bel} is weaker than the assumption of tight resolutions for $N$ and $M$. Note also that since $\mL_iH(H(P,Q))=0$ for every $i\geq 0$ and $P\in \proj{A}, Q\in \proj{B}$, there exists an explicit formula for the global dimension of $\Lambda_{(0,0)}$, see Corollary $7.17$ of \cite{Bel}. 

\subsection{Trivial Extensions of Artin Algebras}
The main property of the assumption that $M$ has a $B$-tight projective
$\Lambda_{(0,0)}$-resolution and $N$ has an $A$-tight projective
$\Lambda_{(0,0)}$-resolution is that $\pd_{\Lambda_{(0,0)}}(0,M,0,0)=\pd_B{M}$ and $\pd_{\Lambda_{(0,0)}}(N,0,0,0)=\pd_A{N}$. Our aim in this subsection is to prove a version of Theorem \ref{MNtightthm} for a trivial extension of Artin algebras $\Lambda=A\ltimes N$. We start by recalling some basic facts for trivial extensions. We refer to \cite{FGR} for more details. 

Let $\Lambda=A\ltimes N$ be a trivial extension of Artin algebras. The objects of $\smod{\Lambda}$ are pairs $(X,f)$ where $X\in \smod{A}$ and $f\colon N\otimes_A X\lxr X$ is an $A$-morphism such that $N{\otimes_A}f\circ f=0$. A morphism $a\colon (X,f)\lxr (Y,g)$ is an $A$-morphism $a\colon X\lxr Y$ such that $f\circ a=N{\otimes_A}a\circ g$. We recall also the following functors. The functor $\mt\colon \smod{A}\lxr \smod{\Lambda}$ is defined by
$\mt(X)=(X\oplus (N{\otimes_A}X),t_X)\in \smod{\Lambda}$ on the $A$-modules $X$, where 
$t_X=\bigl(\begin{smallmatrix}
0 & \iden_{N\otimes X} \\
0 & 0
\end{smallmatrix}\bigr)\colon (N{\otimes_A}X)\oplus (N{\otimes_A}N{\otimes_A}X)\lxr X\oplus (N{\otimes_A}X)$ and given an $A$-morphism $a\colon X\lxr Y$ then
$\mt(a)=\bigl(\begin{smallmatrix}
           a & 0 \\
           0 & F(a) \\
         \end{smallmatrix}\bigr)\colon \mt(X)\lxr \mt(Y)$ is a $\Lambda$-morphism. 
The functor $\mz\colon \smod{A}\lxr \smod{\Lambda}$ is defined by
$\mz(X)=(X,0)\in \smod{\Lambda}$ on the $A$-modules $X$ and if $a\colon X\lxr Y$ is an $A$-morphism, then $\mz(a)=a$. 

We need the following result which is the analogue of Lemma \ref{pdlemma}.

\begin{lem} Let $\Lambda=A\ltimes N$ be a trivial extension of Artin algebras and let $X$ be an $A$-module. Then$\colon$
\[
\pd_{\Lambda}\mz(X)\le 1 +
\maxx\lbrace\pd_{\Lambda}\mz(\Omega^1_A(X)),\pd_{\Lambda}\mz(N)\rbrace 
\]
\end{lem}
\begin{proof} 
Let $\alpha\colon P\lxr X$ be a projective $A$-cover of $X$ with
kernel $\Omega^1_A(X)$. Then we have a short exact sequence of $\Lambda$-modules
\[
\xymatrix{
 & (N\otimes_A\Omega^1_A(X))\oplus(N\otimes_AN\otimes_AP) \ar[d]^{\bigl(\begin{smallmatrix}
0 & N\otimes k \\
0 & 0
\end{smallmatrix}\bigr)} \ar[rr]^{\bigl(\begin{smallmatrix}
N\otimes k & 0 \\
0 & 1
\end{smallmatrix}\bigr)  } && (N\otimes_AP)\oplus (N\otimes_AN\otimes_AP)  \ar[d]_{\bigl(\begin{smallmatrix}
0 & 1 \\
0 & 0
\end{smallmatrix}\bigr)} \ar@{->>}[rr]^{ \ \ \ \ \ \bigl(\begin{smallmatrix}
N\otimes \alpha \\
0
\end{smallmatrix}\bigr) } &&
N\otimes_AX \ar[d]^{0}                   &      \\
     &  \Omega^1_A(X)\oplus(N\otimes_AP) \ \ar@{>->}[rr]^{\bigl(\begin{smallmatrix}
k & 0 \\
0 & 1
\end{smallmatrix}\bigr)} && P\oplus(N\otimes_AP) \ar@{->>}[rr]^{ \ \ \ \ \ \bigl(\begin{smallmatrix}
 \alpha \\
 0
\end{smallmatrix}\bigr) }  &&  X &    }
\]
in which the middle term is a projective $\Lambda$-module. It
follows that 
\[
\pd_{\Lambda}\mz(X)\le 1+
\pd_{\Lambda}\big((N\otimes_A\Omega^1_A(X))\oplus (N\otimes_AN\otimes_AP)\lxr \Omega^1_A(X)\oplus (N\otimes_AP)\big) \eqno(1)
\]  
Next we note that we have the following exact commutative diagram
\[
\xymatrix{
 & N\otimes_AN\otimes_AP \ar[d]^{0} \ \ar@{>->}[rr]^{\bigl(\begin{smallmatrix}
0 & 1
\end{smallmatrix}\bigr) \ \ \ \ \ \ \ \  } && (N\otimes_A\Omega^1_A(X))\oplus(N\otimes_AN\otimes_AP)  \ar[d]_{\bigl(\begin{smallmatrix}
0 & N\otimes k \\
0 & 0
\end{smallmatrix}\bigr)} \ar@{->>}[rr]^{ \ \ \ \ \ \ \ \bigl(\begin{smallmatrix}
1 \\
0
\end{smallmatrix}\bigr) } &&
N\otimes_A \Omega_A^1(X) \ar[d]^{0}                      \\
     & N\otimes_AP \ \ar@{>->}[rr]^{\bigl(\begin{smallmatrix}
0 & 1
\end{smallmatrix}\bigr) \ \ \ \ \ \ \ \ } && \Omega_A^1(X)\oplus(N\otimes_AP) \ar@{->>}[rr]^{ \ \ \ \ \ \ \ \bigl(\begin{smallmatrix}
 1 \\
 0
\end{smallmatrix}\bigr) }  &&  \Omega_A^1(X)      }
\]
and so we have 
\[
\pd_{\Lambda}\big((N\otimes_A\Omega^1_A(X))\oplus (N\otimes_AN\otimes_AP)\to \Omega^1_A(X)\oplus (N\otimes_AP)\big)  \le \maxx\{
\pd_{\Lambda}\mz(N\otimes_AP), \pd_{\Lambda}\mz(\Omega_A^1(X)) \} \ (2)  
\]
Since we have that $N\otimes_AP$ is direct sum of summands of $N$ it follows that $\pd_{\Lambda}\mz(N\otimes_AP) \le \pd_{\Lambda}\mz(N)$. Hence the result follows from the relations $(1)$ and $(2)$.
\end{proof}

We have the following result and its consequence.

\begin{prop}
Let $\Lambda=A\ltimes N$ be a trivial extension of Artin algebras. Then$\colon$ 
\[
\gld{\Lambda} \ \le \ \gld{A}+\pd_{\Lambda}\mz(N)+1
\]
\begin{proof}
Let $X$ be an $A$-module. We will first prove that 
\[
\pd_{\Lambda}\mz(X)\le \pd_AX+
\pd_{\Lambda}\mz(N)+1  \eqno(*)
\]
If $\pd_AX$ is not finite, then the result
follows. Assume that $\pd_AX=n$. If we apply Lemma $5.18$ first to
$\mz(X)$ and then to $\mz(\Omega^1_A(X))$, we get $\pd_{\Lambda}\mz(X)\le 2 +
\maxx\lbrace\pd_{\Lambda}\mz(\Omega^2_A(X)),\pd_{\Lambda}\mz(N)\rbrace$. 
Continuing in this fashion, we obtain 
\[
\pd_{\Lambda}\mz(X)\le n +
\maxx\lbrace\pd_{\Lambda}\mz(\Omega^n_A(X)),\pd_{\Lambda}\mz(N)\rbrace
\]
By assumption, $\Omega^n_A(X)$ is a projective $A$-module.
Applying again Lemma $5.18$ to $\mz(\Omega^n_A(X))$ we obtain the relation $(*)$. Recall from \cite{FGR} that the simple $\Lambda$-modules are of the form $\mz(S)$, where $S$ is a simple $A$-module. Then from the relation $(*)$ we have $\pd_{\Lambda}\mz(S)\le \pd_AS+\pd_{\Lambda}\mz(N)+1$. Thus $\pd_{\Lambda}\mz(S)\le \gld{A}+\pd_{\Lambda}\mz(N)+1$ and so the result follows.
\end{proof}
\end{prop}

\begin{cor}
Let $\Lambda=A\ltimes N$ be a trivial extension of Artin algebras such that $\pd_AN=\pd_{\Lambda}\mz(N)$. Then$\colon$ 
\[
\gld{\Lambda} \ \le \ 2\cdot \gld{A}+1
\]
\end{cor}

\section{Gorenstein Artin Algebras}
In this section we investigate when a Morita ring, which is an Artin algebra, is Gorenstein. Moreover we determine the Gorenstein-projective modules over the matrix algebra with $A=M=N=B=\Lambda$, where $\Lambda$ is an Artin algebra. Recall from \cite{AR:applications}, \cite{AR:cm} that an Artin algebra $\Lambda$ is called \textsf{Gorenstein} if $\id {_\Lambda\Lambda}<\infty$ and $\id\Lambda_{\Lambda}<\infty$. Equivalently $(\proj{\Lambda})^{<\infty}=(\inj{\Lambda})^{<\infty}$, where $(\proj{\Lambda})^{<\infty}$, resp. $(\inj{\Lambda})^{<\infty}$, is the full subcategory of $\smod{\Lambda}$ consisting of the $\Lambda$-modules of finite projective, resp. injective, dimension.

We start with the next result which describes the left derived functors of $\mt_A$, $\mt_B$ and gives also some useful isomorphisms for $\Ext$ homology groups. 

\begin{lem}
Let $\Lambda_{(\phi,\psi)}$ be a Morita ring.
\begin{enumerate}
\item For every $n\geq 1$ we have the following natural isomorphisms$\colon$ 
\[
\xymatrix{
 \mU_{B}\mL_n\mt_{A}(-)  \ \ar[r]^{\simeq  } & \ \Tor_n^A(M,-) 
 } \ \ \ \ \text{and} \ \ \ \ \ \mU_{A}\mL_n\mt_{A}(-)=0 
\]

\item For every $n\geq 1$ we have the following natural isomorphisms$\colon$
\[
\xymatrix{
 \mU_{A}\mL_n\mt_{B}(-)  \ \ar[r]^{\simeq } & \ \Tor_n^B(N,-)  
 } \ \ \ \ \text{and} \ \ \ \ \ \mU_{B}\mL_n\mt_{B}(-)=0
\]

\item If $\Tor_i^A(M,X)=0$, $\forall 1\leq i\leq n$, then we have an isomorphism$\colon$
\[
\xymatrix{
 {\Ext}_{\Lambda_{(\phi,\psi)}}^{i}(\mt_{A}(X),(X',Y',f',g')) \ \ar[r]^{ \ \ \ \ \ \ \ \ \ \ \ \simeq} & \ {\Ext}_{A}^{i}(X,X') }
\]
for every $1\leq i\leq n$ and $(X',Y',f',g')\in \smod{\Lambda_{(\phi,\psi)}}$.

\item If $\Tor_i^B(N,Y)=0$, $\forall 1\leq i\leq n$, then we have an isomorphism$\colon$
\[
\xymatrix{
 {\Ext}_{\Lambda_{(\phi,\psi)}}^{i}(\mt_{B}(Y),(X',Y',f',g')) \ \ar[r]^{ \ \ \ \ \ \ \ \ \ \ \ \simeq} & \ {\Ext}_{B}^{i}(Y,Y') }
\]
for every $1\leq i\leq n$ and $(X',Y',f',g')\in \smod{\Lambda_{(\phi,\psi)}}$.
\end{enumerate}
\begin{proof}
(i) Let $X$ be an $A$-module and let $0 \lxr K_0 \stackrel{i_0}{\lxr} P_0 \stackrel{a_0}{\lxr} X \lxr 0 $ be an exact sequence with $P_0$ a projective $A$-module and $K_0=\Ker{a_0}$. Since $\mL_1\mt_{A}(P_0)=0$ we derive the following exact sequence$\colon$
\[
\xymatrix{
  0 \ar[r] & \mL_1\mt_{A}(X) \ar[r]^{} & \mt_{A}(K_0) \ar[r]^{\mt_{A}(i_0) \ } & \mt_{A}(P_0) \ar[r]^{\mt_{A}(a_0) \ } & \mt_{A}(X) \ar[r] & 0 } 
\]  
Then we have $\mL_1\mt_{A}(X)\simeq \Ker{\mt_{A}(i_0)}=(0,\Ker{(M\otimes_Ai_0)},0,0)\simeq(0,\Tor_1^A(M,X),0,0)$. Continuing as above we infer that
$\mL_n\mt_{A}(X)\simeq (0,\Tor_n^A(M,X),0,0 ), \ \forall \, n\geq 1$, and then (i) follows. Similarly we show that $\mL_n\mt_{B}(Y)\simeq (\Tor_n^B(N,Y),0,0,0 )$ and hence we deduce (ii).

(iii) Let $X$ be an $A$-module and let $\cdots \lxr P_1 \lxr P_0 \lxr X \lxr 0$ be a projective resolution of $X$. Since $\Tor_i^A(M,X)=0$, for
$1\leq i\leq n$, it follows from (i) that $\mL_i\mt_{A}(X)=0$ for every $1\leq i\leq n$. This implies that the sequence $\cdots \lxr \mt_{A}(P_{n+1}) \lxr \mt_{A}(P_n) \lxr \cdots \lxr \mt_{A}(P_0) \lxr \mt_{A}(X) \lxr 0$ is part of a projective resolution of $\mt_{A}(X)$. Let $(X',Y',f',g')$ be a $\Lambda_{(\phi,\psi)}$-module. Then using the adjoint pair $(\mt_{A},\mU_{A})$ we have the following commutative diagram$\colon$
\[  
\xymatrix{
  (\mt_A(X),(X',Y',f',g')) \ar[d]_{\simeq} \ \ar@{>->}[r]^{} & (\mt_A(P_0),(X',Y',f',g')) \ar[d]_{\simeq}
  \ar[r]^{} & (\mt_A(P_1),(X',Y',f',g')) \ar[d]_{\simeq} \ar[r]^{} & \cdots  \\
  \Hom_{A}(X,X') \ \ar@{>->}[r]^{} & \Hom_{A}(P_0,X') \ar[r]^{} & \Hom_{A}(P_1,X') \ar[r]^{} & \cdots   } 
\]
Hence we have the isomorphism ${\Ext}_{\Lambda_(\phi,\psi)}^i(\mt_{A}(X),(X',Y',f',g'))\simeq
{\Ext}_{A}^i(X,X')$ for every $1\leq i\leq n$. Similarly using (ii) we get the isomorphism of (iv).
\end{proof}
\end{lem}

The following main result of this section gives a sufficient condition for a Morita ring to be Gorenstein. 

\begin{thm}\label{GorensteinMoritaring}
Let $\Lambda_{(\phi,\psi)}$ be a Morita ring which is an Artin algebra such that the adjoint pair of functors $(M\otimes_A-,\Hom_B(M,-))$ induces an equivalence
\[
\xymatrix@C=0.5cm{
M\otimes_A-\colon (\proj{A})^{<\infty} \ \ar@<-.7ex>[rr]_-{} && \ (\inj{B})^{<\infty} \ \colon \Hom_B(M,-) \ \ar@<-.7ex>[ll]_{\simeq  \ \ \ }}
\]
and the adjoint pair of functors $(N\otimes_B-,\Hom_A(N,-))$ induces an equivalence
\[
\xymatrix@C=0.5cm{
N\otimes_B-\colon (\proj{B})^{<\infty} \ \ar@<-.7ex>[rr]_-{} && \ (\inj{A})^{<\infty} \ \colon \Hom_A(N,-) \ \ar@<-.7ex>[ll]_{\simeq  \ \ \ }}
\]
Then the Morita ring $\Lambda_{(\phi,\psi)}$ is Gorenstein.
\begin{proof}
We will show that $(\proj{\Lambda_{(\phi,\psi)}})^{<\infty}=(\inj{\Lambda_{(\phi,\psi)}})^{<\infty}$. In oder to prove our claim it suffices to show that any projective $\Lambda_{(\phi,\psi)}$-module has finite injective dimension and any injective $\Lambda_{(\phi,\psi)}$-module has finite projective dimension. Thus from Proposition \ref{proj} and Proposition \ref{inj} it suffices to show that ${\id}_{\Lambda_{(\phi,\psi)}}\mt_A(P)<\infty$, ${\id}_{\Lambda_{(\phi,\psi)}}\mt_B(Q)<\infty$, ${\pd}_{\Lambda_{(\phi,\psi)}}\mh_A(I)<\infty$ and ${\pd}_{\Lambda_{(\phi,\psi)}}\mh_B(J)<\infty$ for any $P\in \proj{A}$, $Q\in \proj{B}$, $I\in \inj{A}$ and $J\in \inj{B}$. Let $I\in \inj{A}$. By hypothesis the counit $\varepsilon'_I\colon N\otimes_B\Hom_A(N,I)\lxr I$ is an isomorphism and consider the $\Lambda_{(\phi,\psi)}$-modules $\mh_{A}(I)=(I,\Hom_A(N,I),\delta'_{M\otimes I}\circ \Hom_A(N,\Psi_I),\epsilon'_I)$ and $\mt_B(\Hom_A(N,I))=(N\otimes_B\Hom_A(N,I),\Hom_A(N,I),\Phi_{\Hom_A(N,I)},\iden_{N\otimes_B\Hom_A(N,I)})$. Since $\mh_{A}(I)$ is a $\Lambda_{(\phi,\psi)}$-module we have the following commutative diagram$\colon$
\[
\xymatrix{
  M\otimes_A N\otimes_B \Hom_A(N,I) \ar[d]_{\phi\otimes \iden_{\Hom_A(N,I)}} \ar[r]^{ \ \ \ \ \ \ \ \ \ \ M\otimes \varepsilon'_I } &  M\otimes_AI \ar[d]^{\delta'_{M\otimes I}\circ \Hom_A(N,\Psi_I)}     \\
  B\otimes_B \Hom_A(N,I)    \ar[r]^{ \ \ \simeq} & \Hom_A(N,I)                  }
\]
and therefore we have the following map$\colon$
\[
(\varepsilon'_I, \iden_{\Hom_A(N,I)})\colon \mt_B(\Hom_A(N,I))\lxr \mh_A(I) \eqno(*)
\]
which is an isomorphism of $\Lambda_{(\phi,\psi)}$-modules. Since $I$ is an injective $A$-module it follows that the $B$-module $\Hom_A(N,I)$ has finite projective dimension. Let $0\lxr P_n\lxr\cdots \lxr P_1\lxr P_0\lxr \Hom_A(N,I)\lxr 0$ be a projective resolution of $\Hom_A(N,I)$ in $\smod{B}$. Since the functor $N\otimes_B-$ is an equivalence restricted to the subcategory $(\proj{B})^{<\infty}$ it follows that the complex
$0 \lxr N\otimes_BP_n \lxr \cdots \lxr N\otimes_BP_0 \lxr N\otimes_B\Hom_A(N,I) \lxr 0$ is exact. This implies that $\Tor_n^{B}(N,\Hom_A(N,I))=0$, $\forall n\geq 1$, and then from Lemma $6.1$ we have the following isomorphism$\colon$
\[
\xymatrix@C=0.5cm{
   \Ext^n_{\Lambda_{(\phi,\psi)}}\big(\mt_B(\Hom_A(N,I)), (X,Y,f,g) \big) \ar[rr]^{ \ \ \ \ \ \ \ \ \ \ \simeq} && \Ext_{B}^n\big(\Hom_A(N,I),Y\big)   }
\]
for every $n\geq 1$ and $(X,Y,f,g)\in \smod{\Lambda_{(\phi,\psi)}}$. Since $\pd_B\Hom_A(N,I)<\infty$ it follows from the above isomorphism that $\pd_{\Lambda_{(\phi,\psi)}}\mt_B(\Hom_A(N,I))<\infty$. Hence from the relation $(*)$ we infer that the projective dimension of $\mh_A(I)$ is finite. Similarly we prove that ${\id}_{\Lambda_{(\phi,\psi)}}\mt_A(P)<\infty$, ${\id}_{\Lambda_{(\phi,\psi)}}\mt_B(Q)<\infty$ and ${\pd}_{\Lambda_{(\phi,\psi)}}\mh_B(J)<\infty$. We infer that $(\proj{\Lambda_{(\phi,\psi)}})^{<\infty}=(\inj{\Lambda_{(\phi,\psi)}})^{<\infty}$ and therefore the Morita ring $\Lambda_{(\phi,\psi)}$ is Gorenstein.
\end{proof}
\end{thm}

\begin{rem}
\begin{enumerate}
\item Let $\Lambda_{(\phi,\psi)}$ be a Morita ring and assume as above that the adjoint pair of functors $(M\otimes_A-,\Hom_B(M,-))$ induces quasi-inverse equivalences between $(\proj{A})^{<\infty}$ and $(\inj{B})^{<\infty}$, and the adjoint pair of functors $(N\otimes_B-,\Hom_A(N,-))$ induces quasi-inverse equivalences between $(\proj{B})^{<\infty}$ and $(\inj{A})^{<\infty}$. Since $A\in (\proj{A})^{<\infty}$ it follows that $\id {_BM}<\infty$ and $A\simeq \End_B(M)$, and similarly since $B\in (\proj{B})^{<\infty}$ we get that $\id {_AN}<\infty$ and $B\simeq \End_A(N)$. 

\item Since selfinjective algebras are Gorenstein, it follows from Example $3.9$ that the converse of Theorem \ref{GorensteinMoritaring} is not true in general.
\end{enumerate}
\end{rem}

If $\Lambda_{(\phi,\psi)}$ is a Morita ring with $A=M=N=B$ then we know from Corollary \ref{equal} that the bimodule homomorphisms $\phi$ and $\psi$ are equal. From now on we denote the Morita ring with all entries a ring $\Lambda$ by $\Delta_{(\phi,\phi)}=\bigl(\begin{smallmatrix}
\Lambda & \Lambda \\
\Lambda & \Lambda
\end{smallmatrix}\bigr)$. It is known from Fossum-Griffith-Reiten \cite{FGR}, see also Happel \cite{Happel}, that if $\Lambda$ is a Gorenstein Artin algebra then the upper triangular matrix algebra $\bigl(\begin{smallmatrix}
\Lambda & \Lambda \\
\Lambda & 0
\end{smallmatrix}\bigr)$ is Gorenstein. In this connection we have the next result, which is a consequence of Theorem \ref{GorensteinMoritaring}, and shows that $\Delta_{(\phi,\phi)}$\index{$\Delta_{(\phi,\phi)}$} is Gorenstein when $\Lambda$ is as well.

\begin{cor}
Let $\Lambda$ be an Artin algebra. Then $\Lambda$ is Gorenstein if and only if the Morita ring $\Delta_{(\phi,\phi)}$ is Gorenstein Artin algebra.
\begin{proof}
Suppose that $\Lambda$ is Gorenstein. Then we have $(\proj{\Lambda})^{<\infty}=(\inj{\Lambda})^{<\infty}$ and so from Theorem \ref{GorensteinMoritaring} it follows that matrix algebra $\Delta_{(\phi,\phi)}$ is Gorenstein. Conversely assume that $\Delta_{(\phi,\phi)}$ is Gorenstein and let $I$ be an injective $\Lambda$-module. Then the injective $\Delta_{(\phi,\phi)}$-module $\mh_{\Lambda}(I)$ has finite projective dimension since $(\proj{\Delta_{(\phi,\phi)}})^{<\infty}=(\inj{\Delta_{(\phi,\phi)}})^{<\infty}$. Consider the exact sequence $\cdots \lxr \mt_{\Lambda}(P_1)\oplus \mt_{\Lambda}(Q_1) \lxr \mt_{\Lambda}(P_0)\oplus \mt_{\Lambda}(Q_0) \lxr \mh_{\Lambda}(I) \lxr 0$
which is the start of a finite projective resolution of $\mh_{\Lambda}(I)$. Then applying the functor $\mU_{\Lambda}\colon \smod{\Delta_{(\phi,\phi)}}\lxr \smod{\Lambda}$ we obtain the exact sequence $\cdots \lxr P_1\oplus Q_1 \lxr P_0\oplus Q_0 \lxr I \lxr 0$ and this implies that $\pd{_{\Lambda}I}<\infty$. Similarly we show that $\id{_{\Lambda}P}<\infty$ for every $P\in \proj{\Lambda}$. We infer that $(\proj{\Lambda})^{<\infty}=(\inj{\Lambda})^{<\infty}$ and therefore the Artin algebra $\Lambda$ is Gorenstein.
\end{proof}
\end{cor}

Recall that an acyclic complex of projective $\Lambda$-modules $\textsf{P}^{\bullet}\colon \cdots \lxr P^{i-1}\lxr P^i\lxr P^{i+1}\lxr \cdots$ is called \textsf{totally acyclic}, if the complex $\Hom_{\Lambda}(\textsf{P}^{\bullet},\Lambda)$ is acyclic. Then a $\Lambda$-module $X$ is called \textsf{Gorenstein-projective} if it is of the form $X=\Coker{(P^{-1}\lxr P^0)}$ for some totally acyclic complex $\textsf{P}^{\bullet}$ of projective $\Lambda$-modules. For an Artin algebra $\Lambda$ we denote by $\Gproj{\Lambda}$ the full subcategory of $\smod{\Lambda}$ consisting of the finitely generated Gorenstein-projective $\Lambda$-modules. It is well known, see \cite[Proposition $3.10$]{Bel: virtually}, that when $\Lambda$ is Gorenstein then $\Gproj{\Lambda}=\{X \ | \ \Ext_{\Lambda}^n(X,\Lambda)=0, \ \forall n\geq 1 \}$. Finally recall from \cite{Bel: virtually}, \cite{Bel: finite CM type} that an Artin algebra $\Lambda$ is said to be of \textsf{finite Cohen-Macaulay type} if the category $\Gproj{\Lambda}$ of finitely generated Gorenstein-projective $\Lambda$-modules is of finite representation type, i.e. the set of isomorphism classes of its indecomposable objects is finite.

Recently Li and Zhang \cite{Zhang} determined the Gorenstein-projective modules over the triangular matrix algebra $\bigl(\begin{smallmatrix}
\Lambda & \Lambda \\
0 & \Lambda
\end{smallmatrix}\bigr)$, when $\Lambda$ is a Gorenstein Artin algebra, and using this they obtained a criterion for the Cohen-Macaulay finiteness of $\bigl(\begin{smallmatrix}
\Lambda & \Lambda \\
0 & \Lambda
\end{smallmatrix}\bigr)$ in case that $\Lambda$ is a Gorenstein Artin algebra of finite Cohen-Macaulay type. Our aim now is to describe the Gorenstein-projective modules over the algebra $\Delta_{(\phi,\phi)}$. For the ring $\Delta_{(\phi,\phi)}$ we denote by $\mt_{\Lambda}^{\prime}$, resp. $\mh_{\Lambda}^{\prime}$, the functor $\mt_B$, resp. $\mh_{B}$, where the algebra $B$ is now $\Lambda$. 

We need the following observation.

\begin{lem}
Let $\Delta_{(\phi,\phi)}$ be a Morita ring. Then we have isomorphisms of functors$\colon$ $\mt_{\Lambda}(-)\simeq \mh^{\prime}_{\Lambda}(-)$ and $\mt_{\Lambda}^{\prime}(-)\simeq \mh_{\Lambda}(-)$.
\begin{proof}
Let $X$ be a $\Lambda$-module and $f\colon \Hom_{\Lambda}(\Lambda,X)\stackrel{\simeq}{\lxr}X$, $g\colon \Lambda\otimes_{\Lambda}X\stackrel{\simeq}{\lxr} X$ the standard isomorphisms. Then it is easy to check that the maps $(f^{-1},g)\colon \mt_{\Lambda}(X)=(X,\Lambda\otimes_{\Lambda}X,\iden_{\Lambda\otimes X},\Phi_X)\lxr \mh^{\prime}_{\Lambda}(X)=(\Hom_{\Lambda}(\Lambda,X),X,\epsilon_X,\delta_{\Lambda\otimes X}\circ \Hom_{\Lambda}(\Lambda,\Phi_X))$ and $(g, f^{-1})\colon \mt_{\Lambda}^{\prime}(X)=(\Lambda\otimes_{\Lambda}X,X,\Phi_X,\iden_{\Lambda\otimes X})\lxr \mh_{\Lambda}(X)=(X,\Hom_{\Lambda}(\Lambda,X),\delta'_{\Lambda\otimes X}\circ \Hom_{\Lambda}(\Lambda,\Phi_X),\epsilon'_X)$ are isomorphisms of $\Delta_{(\phi,\phi)}$-modules. 
\end{proof}
\end{lem}

The following result characterizes when a module over the algebra $\Delta_{(\phi,\phi)}$ is Gorenstein-projective.

\begin{cor}
Let $\Lambda$ be a Gorenstein Artin algebra. Then a $\Delta_{(\phi,\phi)}$-module $(X,Y,f,g)$ is Gorenstein-projective if and only if $X$ and $Y$ are Gorenstein-projective $\Lambda$-modules.
\begin{proof}
Let $0\lxr \Lambda\lxr I_{0}\lxr I_1\lxr \cdots \lxr I_n\lxr 0$ be an injective coresolution of $_{\Lambda}{\Lambda}$. If we apply the functors $\mh_{\Lambda}, \mh^{\prime}_{\Lambda}\colon \smod{\Lambda}\lxr \smod{\Delta_{(\phi,\phi)}}$ we obtain the exact sequences $0\lxr \mh_{\Lambda}(\Lambda)\lxr \mh_{\Lambda}(I_{0})\lxr \cdots \lxr \mh_{\Lambda}(I_n)\lxr 0$ and $0\lxr \mh^{\prime}_{\Lambda}(\Lambda)\lxr \mh^{\prime}_{\Lambda}(I_{0})\lxr \cdots \lxr \mh^{\prime}_{\Lambda}(I_n)\lxr 0$ which are injective coresolutions of $\mh_{\Lambda}(\Lambda)$ and $\mh^{\prime}_{\Lambda}(\Lambda)$ respectively. From Lemma $6.5$ the above resolutions can be regarded as injective coresolutions of $\mt^{\prime}_{\Lambda}(\Lambda)$ and $\mt_{\Lambda}(\Lambda)$. 
Then using the adjoint pairs of functors $(\mU^{\prime}_{\Lambda},\mh_{\Lambda}^{\prime})$ and  $(\mU_{\Lambda},\mh_{\Lambda})$ we have the following commutative diagrams$\colon$
\[  
\xymatrix{
  0 \ar[r] & \big((X,Y,f,g),\mt_{\Lambda}(\Lambda)\big) \ar[d]_{\simeq}
  \ar[r]^{} & \big((X,Y,f,g),\mh_{\Lambda}^{\prime}(I_0)\big) \ar[d]_{\simeq} \ar[r]^{} & \cdots \ar[r]^{} & \big((X,Y,f,g),\mh_{\Lambda}^{\prime}(I_n)\big) \ar[d]_{\simeq} \ar[r] & 
  0 \\
 0 \ar[r] & \Hom_{\Lambda}(Y,\Lambda) \ar[r]^{} & \Hom_{\Lambda}(Y,I_0) \ar[r]^{} & \cdots \ar[r]^{} & \Hom_{\Lambda}(Y,I_n)
  \ar[r]^{} & 0   } 
\]
and
\[  
\xymatrix{
  0 \ar[r] & \big((X,Y,f,g),\mt_{\Lambda}^{\prime}(\Lambda)\big) \ar[d]_{\simeq}
  \ar[r]^{} & \big((X,Y,f,g),\mh_{\Lambda}(I_0)\big) \ar[d]_{\simeq} \ar[r]^{} & \cdots \ar[r]^{} & \big((X,Y,f,g),\mh_{\Lambda}(I_n)\big) \ar[d]_{\simeq} \ar[r] & 
  0 \\
 0 \ar[r] & \Hom_{\Lambda}(X,\Lambda) \ar[r]^{} & \Hom_{\Lambda}(X,I_0) \ar[r]^{} & \cdots \ar[r]^{} & \Hom_{\Lambda}(X,I_n)
  \ar[r]^{} & 0   } 
\]
These diagrams imply that $\Ext^{n}_{\Lambda}(Y,\Lambda)=0$, $\forall n\geq 1$, if and only if $\Ext^n_{\Delta_{(\phi,\phi)}}((X,Y,f,g),\mt_{\Lambda}(\Lambda))=0$, $\forall n\geq 1$, and $\Ext^{n}_{\Lambda}(X,\Lambda)=0$, $\forall n\geq 1$, if and only if $\Ext^n_{\Delta_{(\phi,\phi)}}((X,Y,f,g),\mt_{\Lambda}^{\prime}(\Lambda))=0$, $\forall n\geq 1$. Since $\Delta_{(\phi,\phi)}\simeq \mt_{\Lambda}(\Lambda)\oplus \mt_{\Lambda}^{\prime}(\Lambda)$ as $\Delta_{(\phi,\phi)}$-modules it follows from Corollary $6.4$ that $(X,Y,f,g)\in \Gproj{\Delta_{(\phi,\phi)}}$ if and only if $X, Y\in \Gproj{\Lambda}$.
\end{proof}
\end{cor}

After the above description it would be interesting to examine when the matrix algebra $\Delta_{(\phi,\phi)}$ is of finite Cohen-Macaulay type \cite{Bel: finite CM type}.

We close this section with the following result which gives the connection between the category of Gorenstein-projective modules over $\Delta_{(\phi,\phi)}$ and the corresponding category of $\Lambda$.

\begin{cor}
Let $\Lambda$ be a Gorenstein Artin algebra. Then the recollement situation of $\smod{\Delta_{(\phi,\phi)}}$ is restricted to the categories of Gorenstein-projective modules $\Gproj{\Delta_{(\phi,\phi)}}$ and $\Gproj{\Lambda}$.
\begin{proof}
Let $X$ be a Gorenstein-projective $\Lambda$-module. Since $\Lambda$ is Gorenstein there exists an exact sequence $0\lxr X\lxr P^0\lxr P^1\lxr \cdots$ where each $P^i$ is a projective $\Lambda$-module. If we apply the exact functors $\mt_{\Lambda}$ and $\mt^{\prime}_{\Lambda}$ we get that $\mt_{\Lambda}(X)$ and $\mt_{\Lambda}^{\prime}(X)$ are Gorenstein-projective $\Delta_{(\phi,\phi)}$-modules, since from Corollary $6.4$ the Artin algebra $\Delta_{(\phi,\phi)}$ is Gorenstein. Also from Corollary $6.6$ it follows that $\mU_{\Lambda}(X,Y,f,g)$ is Gorenstein-projective for every $(X,Y,f,g)\in \Gproj{\Delta_{(\phi,\phi)}}$ and finally we have $\Ker{\mU_{\Lambda}}=\{(0,Y,0,0)\in  \Gproj{\Delta_{(\phi,\phi)}} \ | \ Y\in \Gproj{\Lambda} \}$.  
\end{proof}
\end{cor} 

\begin{ackn}
This work started when both authors were visiting the University of Bielefeld in $2010$. The authors would like to express their special thanks to Henning Krause for the warm hospitality and the excellent working conditions. This work has been completed during a stay of the second author at the Norwegian University of Science and Technology (NTNU, Trondheim) in $2012$. The second author would like to thank ${\O}$yvind Solberg for the invitation and the warm hospitality. Finally the authors are grateful to Apostolos Beligiannis, Magdalini Lada and Jorge Vit\'oria for many helpful discussions and useful comments concerning the material of the paper.
\end{ackn}



\end{document}